\documentclass[11pt]{article}
\usepackage{fullpage}
\usepackage{amsmath,amssymb,amsthm,bbm}
\usepackage{algorithm,algpseudocode}
\usepackage{graphicx,subfigure}
\usepackage{threeparttable}
\usepackage{enumerate}
\usepackage{multirow}
\usepackage{url}
\usepackage{bm}
\usepackage{color}
\usepackage{mathtools}
\usepackage{tikz}
\usepackage{array}
\usepackage{tabu}

\graphicspath{{figs/}}

\newcommand{\Romannumber}[1]{\uppercase\expandafter{\romannumeral #1}}

\newcommand{\real}{\mathbb{R}}

\newcommand{\besselk}{\mathsf{K}}

\newcommand{\ud}[1]{\underline{#1}}
\renewcommand{\Re}{\operatorname{Re}}

\DeclareMathOperator{\spn}{span}

\DeclareMathOperator{\diag}{diag}

\DeclareMathOperator{\tr}{tr}
\DeclareMathOperator{\nnz}{nz}

\DeclareMathOperator{\erf}{erf}

\DeclareMathOperator{\Prob}{Pr}

\theoremstyle{definition} 
\theoremstyle{remark}     
\theoremstyle{remark}     
\theoremstyle{plain}      \newtheorem{theorem}{Theorem}%[section]
\theoremstyle{plain}      
\theoremstyle{plain}      \newtheorem{proposition}[theorem]{Proposition}
\theoremstyle{plain}      \newtheorem{corollary}[theorem]{Corollary}
\theoremstyle{plain}

\algblock{Start}{End}

        {\vskip 10pt \begin{algorithmic}[1]}%
        {\end{algorithmic} \vskip 10pt}

\begin{document}

%%%%%%%%%%%%%%%%%%%%%%%%%%%%%%%%%%%%%%%%%%%%%%%%%%%%%%%%%%%%
\title{\textsf{A Posteriori Error Estimate for Computing $\tr(f(A))$ by Using the Lanczos Method}}
\author{Jie Chen\thanks{IBM Thomas J. Watson Research Center. Email: \texttt{chenjie@us.ibm.com}}
  \and Yousef Saad\thanks{Department of Computer Science and Engineering, University of Minnesota. Email: \texttt{saad@umn.edu}}
}
\maketitle

%%%%%%%%%%%%%%%%%%%%%%%%%%%%%%%%%%%%%%%%%%%%%%%%%%%%%%%%%%%%
\begin{abstract}
An outstanding problem when computing a function of a matrix, $f(A)$,
by using a Krylov method is to accurately estimate errors when
convergence is slow.  Apart from the case of the exponential function
which has been extensively studied in the past, there are no
well-established solutions to the problem.  Often the quantity of
interest in applications is not the matrix $f(A)$ itself, but rather,
matrix-vector products or bilinear forms. When the computation related
to $f(A)$ is a building block of a larger problem (e.g., approximately
computing its trace), a consequence of the lack of reliable error
estimates is that the accuracy of the computed result is unknown. In
this paper, we consider the problem of computing $\tr(f(A))$ for a
symmetric positive-definite matrix $A$ by using the Lanczos method and
make two contributions: (i) we propose an error estimate for the
bilinear form associated with $f(A)$, and (ii) an error estimate for
the trace of $f(A)$.  We demonstrate the practical usefulness of these
estimates for large matrices and in particular, show that the trace
error estimate is indicative of the number of accurate digits. As an
application, we compute the log-determinant of a covariance matrix in
Gaussian process analysis and underline the importance of error
tolerance as a stopping criterion, as a means of bounding the number
of Lanczos steps to achieve a desired accuracy.
\end{abstract}

%%%%%%%%%%%%%%%%%%%%%%%%%%%%%%%%%%%%%%%%%%%%%%%%%%%%%%%%%%%%
\section*{keywords}
Matrix trace, matrix function, Lanczos method, error estimate, 
confidence interval

%%%%%%%%%%%%%%%%%%%%%%%%%%%%%%%%%%%%%%%%%%%%%%%%%%%%%%%%%%%%
\section{Introduction}
The trace  of a function of  a matrix, $\tr(f(A))$, occurs  in diverse
areas   including  scientific   computing,  statistics,   and  machine
learning~\cite{Bai1996,Golub1997,Stein1999,Estrad2000,Rasmussen2006,Bekas2007,GP-SAA,score.func,Stathopoulos2013,Han2015,Wu2016,Lininpress}. Often in applications, the matrix $A$ is so large that explicitly forming
$f(A)$ is not  a practically viable option. In this  work, we focus on
the case when $A\in\real^{n\times  n}$ is symmetric positive-definite,
so that      it     admits     a      spectral     decomposition
$Q^TAQ=\diag(\lambda_1,\ldots,\lambda_n)$, where the $\lambda_i$'s are
real  positive  eigenvalues  and  $Q$  is  the  matrix  of  normalized
eigenvectors.
Naturally, $f$ must be (at least) defined on the spectrum of $A$, although for analysis, we will assume that $f$ is analytic inside some contour enclosing the spectrum.
Then,  $\tr(f(A))$  is  nothing  but  the  sum  of  the
$f(\lambda_i)$'s. Computing the trace in this manner, however, requires
the  computation   of  all   the  eigenvalues,   which  is   also often
prohibitively   expensive.  Hence,   various   methods  proposed   for
approximately  computing  $\tr(f(A))$  consist of  the  following  two
ingredients~\cite{Girard1987,Hutchinson1990,Bai1996,Avron2011,Han2015,Lin2015,Han2016,Fika2016,Fika2016a}.
\begin{enumerate}
\item Approximate the trace of $f(A)$ by using the average of unbiased samples $u_i^Tf(A)u_i$, $i=1,\ldots,N$, where the $u_i$'s are independent random vectors of some nature.
\item Approximately compute  the bilinear form $u_i^Tf(A)u_i$ 
by using some numerical technique.
\end{enumerate}

The various methods differ in the  random mechanism of selecting the $u_i$'s and
the  numerical technique  for computing  the bilinear  form.
Several variants of these ingredients exist (e.g., computing deterministically $\tr(f(A))=\sum_{i=1}^ne_i^Tf(A)e_i$ rather than using random vectors $u_i$, or even using block vectors to replace the canonical vectors $e_i$~\cite{Bellalij2015}; or using moment extrapolation for particularly $f(t)=t^{\alpha}$ with real value $\alpha$~\cite{Meurant2009,Brezinski2012}), but they are not the focus of this work.
In  the two ingredients, the convergence of  the approximation is generally gauged through
a combination of Monte Carlo convergence of the sample average and the
convergence  of the  numerical  technique. In  practice, however,  the
convergence               results   obtained for                these
methods~\cite{Avron2011,Wimmer2014,Roosta-Khorasani2015,Han2015,Han2016}
rarely translate  into  practical  schemes to monitor
convergence.  
%%In contrast  to
%% solving linear systems  where the residual is a  computable measure of
%% progress in the case of matrix functions, 
Monitoring  the accuracy of  a given approximation to $\tr(f(A))$ can be
very challenging for certain functions $f$.
This is in complete contrast to the situation prevalent when solving
linear systems, where simple residual norms provide a computable 
measure of the backward error. 
The practical question we woud like to address is given the number  $N$ of random vectors  and a
stopping criterion  for the  bilinear form, how  accurate is
the computed result?
%Or, if a  certain accuracy is desired, how should
%one set  $N$ and what  stopping  criterion should be used? 
%This work aims  at answering
%these questions.

The idea  is to obtain  tight \emph{a-posteriori error  estimates} for
both of the ingredients mentioned above  and then to combine these estimates. 
For  the sample
average, we  will establish  confidence intervals.
% based  on estimation theory  in  statistics.  
The  computed  approximation  (a  \emph{point
  estimate} in  a statistics language)  to $\tr(f(A))$ alone  does not
carry any information on accuracy; however, combined with a confidence
interval  (an   \emph{interval  estimate}),  it  gives   a  notion  of
absolute/relative error with (high) probability. The difference with
standard
statistics, on  the other  hand,  is that each  sample itself  bears a
numerical error. Hence, for the approximation of the bilinear form, we
impose  a stopping  criterion---error tolerance  $\delta$---and inject
$\delta$ into  the confidence  interval. The confidence  interval thus
indicates that with a certain  (high) probability, the trace approximation
error is bounded  by some expression in terms of  $N$ and $\delta$. As
we will demonstrate in experiments, this bound is generally indicative
of the number of accurate digits.

Then, monitoring the error in the approximation of the bilinear form
is crucial for an accurate understanding of the overall error in the
trace. The monitoring mechanism 
must depend on the  approximation method used.  In this work we focus on the Lanczos
method, which  has many appealing properties and which has long been a
preferred technique for approximating full or partial  spectra of large
symmetric matrices, through inexpensive matrix-vector multiplications. An outstanding problem, however, is that
 good extensions of the a-posteriori error estimate 
given in Saad~\cite{Saad1992} to more general functions 
 than the exponential are rare. The Lanczos method can be considered a polynomial
approximation technique, where $f$ is approximated by a polynomial
that interpolates $f$ on the Ritz values, but it typically converges twice as fast
as  other polynomial approximation methods (e.g., Chebyshev
approximation)~\cite[Chapter 19]{Trefethen2012}. Such a faster
convergence is owed to the Gauss-quadrature interpretation that will be
discussed shortly.

Let us briefly review the Lanczos method. It begins with a unit vector $v_1$ and coefficient $\beta_1=0$ and computes the sequence
of vectors 
\begin{equation}\label{eqn:lanczos}
v_{k+1}=(Av_k - \alpha_kv_k - \beta_kv_{k-1})/\beta_{k+1},
\quad\text{for } k =1,2,\ldots,
\end{equation}
where $\alpha_k=v_k^TAv_k$, and $\beta_{k+1}$ is a normalization factor such that $v_{k+1}$ has a unit norm. 
After $m$ steps, the above iteration results in the matrix identity
\begin{equation}\label{eqn:AVVT}
AV_m=V_mT_m+\beta_{m+1}v_{m+1}e_m^T,
\end{equation}
where
\[
V_m=[v_1,v_2,\ldots,v_m] \quad\text{and}\quad
T_m=\begin{bmatrix}
\alpha_1 & \beta_2 \\
\beta_2 & \alpha_2 & \beta_3 \\
& \beta_3 & \alpha_3 & \ddots \\
& & \ddots & \ddots & \beta_m\\
& & & \beta_m & \alpha_m
\end{bmatrix}.
\]
In exact  arithmetic, the columns  of $V_m$, together  with $v_{m+1}$,
consist of  an orthonormal  basis of  the Krylov  subspace $\spn\{v_1,
Av_1,  \ldots, A^{m-1}v_1,  A^mv_1\}$, and  the symmetric  tridiagonal
matrix  $T_m$  is sometimes  called  the  \emph{Jacobi matrix}.  Then,
omitting the  index $i$ in  the random  vector $u_i$ for  clarity, the
Lanczos  method approximates  the bilinear  form $u^Tf(A)u$  through a
projection       on       the       Krylov       subspace;       i.e.,
$u^TV_mf(V_m^TAV_m)V_m^Tu$. If the starting  Lanczos vector $v_1$ is a
normalized $u$, then this quantity is simply $\|u\|^2e_1^Tf(T_m)e_1$.

The same approximate quantity may be derived from a different viewpoint. 
Based on the spectral decomposition of $A$, one may write
\begin{equation}\label{eqn:int}
u^Tf(A)u=\sum_{i=1}^nf(\lambda_i){\omega_i^2=\int f(\lambda)\,d\omega}(\lambda),
\end{equation}
where the $\omega_i$'s are elements of the vector $Q^Tu$, $\omega(\lambda)$
is a discrete measure with masses $\omega_i^2$ at the atoms $\lambda_i$,
and the integral is a Stieltjes integral. One may
show~\cite{Saad2003,Golub2009} that there is a sequence of 
polynomials $p_k(\lambda)$  associated with this process satisfying the relation
$v_k=p_{k-1}(A)v_1$, $k=1,2,\ldots,m$, that are orthonormal with respect to the measure $\omega(\lambda)$. Then,
applying  the  Golub--Welsch   algorithm~\cite{Golub1969},  the  Gauss
quadrature rule for the  integral~\eqref{eqn:int} uses the eigenvalues
of $T_m$  as  the quadrature points,  and the
square of the  first element of the normalized  eigenvectors of $T_m$,
multiplied by $\int d\omega=\|u\|^2$, as the quadrature weights. In other
words,  writing  the   spectral  decomposition  $S^TT_mS=\Theta$,  the
quadrature rule gives
\[
\int f(\lambda)\,d\omega(\lambda)
\approx \sum_{k=1}^m \underbrace{\|u\|^2S_{1k}^2}_{\text{weights}}f(\underbrace{\Theta_{kk}}_{\text{points}}),
\]
which coincides with $\|u\|^2e_1^Tf(T_m)e_1$.

The quadrature interpretation is  particularly useful for establishing
exponential convergence of  the Lanczos method, if  $f$ admits certain
analytic properties (see,  e.g., Trefethen\cite{Trefethen2012}). Interestingly,
other  related quadrature  rules, in  a combined  use, may  also 
yield
bounds~\cite{Golub2009}. For  example, if the even  derivatives of $f$
in the spectrum interval have a constant sign, then the Gauss rule and
the  Gauss--Lobatto rule  always  give  results on  the  two sides  of
$u^Tf(A)u$. Similarly, if the odd  derivatives of $f$ have a constant
sign, then  the two results  of the Gauss--Radau rule  always straddle
around  the bilinear form. Moreover,  bounds  for  the particular case when $f$ is a
rational function%
\footnote{Note  that being  a rational  function is  not a  particular
  restriction,  because rational  approximations  are one  of the  key
  tools  for  computing  matrix  functions. For  more  discussions  on
  rational  approximations,  see Section~\ref{sec:rat.approx}.  For  a
  connection     between     the      error     estimation     methods
  in~\cite{Frommer2009,Frommer2013}  and  ours,   see  the  concluding
  section.}  were also  proposed~\cite{Frommer2009,Frommer2013}, based on similar ideas  of Gauss and
Gauss-related quadratures. In practice,
however, these bounds are often too conservative as an error estimate,
especially when convergence is slow. Thus, a contribution of this work
is   a   more   accurate   error   estimate   of   the   approximation
$\|u\|^2e_1^Tf(T_m)e_1$ to $u^Tf(A)u$. This  estimate is directly used
to check against the  aforementioned tolerance $\delta$ for monitoring
progress.

It  is noteworthy  to relate this work with a  prior
work~\cite{Hutchinson}  by the  first  author, who  studied the  trace
error by using a similar approach,  and to underline a distinction between
the two contributions.  Both works consider the combination
of statistical error caused by  sample average and the numerical error
in  evaluating  the  bilinear  form  $u^T f(A) u$.  In  quantifying  the
statistical error,  the prior work  exploited a variance  term defined
through  an estimator,  which is  applicable only  to random  Gaussian
vectors.  In this  work, the variance  is the  sample variance
(albeit carrying numerical error) and  thus there is no restriction on
the random  mechanism of  the random vectors.  This distinction  has a
consequence on the handling of  numerical error. In order to establish
a confidence  interval, the previous work proposed a  stopping criterion
for the approximation of the function%
\footnote{In  fact, the prior work  also discussed  the special  case
  $f(t)=t^{-1}$,  wherein  the  stopping  criterion  is  cast  on  the
  residual of the linear system instead.}  $f$ such that the numerical
error is comparable with the statistical  error. On the other hand, in
this work,  we allow any  tolerance $\delta$ for the  approximation of
the bilinear  form $u^Tf(A)u$,  because $\delta$  is written  into the
confidence interval. Of  course, one may find  an appropriate $\delta$
that makes the two sources of errors comparable, in a post-hoc manner,
but this benefit comes only as a by-product. Nevertheless, the post-hoc
adjustment of $\delta$ is practically useful; see the next section.

%%%%%%%%%%%%%%%%%%%%%%%%%%%%%%%%%%%%%%%%%%%%%%%%%%%%%%%%%%%%
\section{Confidence interval with numerical error}
To establish a confidence interval that incorporates numerical errors in the samples, let us first define some notation. Denote by
\begin{enumerate}
\item $\mu:=\tr(f(A))$, the mean;
\item $x_i:=u_i^Tf(A)u_i$, an independent, unbiased sample; and
\item $x_i^{(m)}:=\|u_i\|^2e_1^Tf(T_m^{(i)})e_1$, a sample with numerical error,
\end{enumerate}
where we have added a superscript $(i)$ to the Jacobi matrix to distinguish different samples. Formally, the trace approximation method considered in this work refers to approximating the mean $\mu$ by using the sample average
\[
\bar{x}^{(m)}:=\frac{1}{N}\sum_{i=1}^Nx_i^{(m)}.
\]
The corresponding sample standard error is
\[
s^{(m)}:=\sqrt{\frac{1}{N-1}\sum_{i=1}^N\left(x_i^{(m)}-\bar{x}^{(m)}\right)^2}.
\]

In standard statistics, one may establish confidence intervals for only the average of the samples \emph{without} numerical error/bias:
\[
\bar{x}:=\frac{1}{N}\sum_{i=1}^Nx_i.
\]
Hence, defining the standard error
\[
s:=\sqrt{\frac{1}{N-1}\sum_{i=1}^N(x_i-\bar{x})^2},
\]
for any $\alpha>0$, we let
\begin{equation}\label{eqn:conf.intv}
p_{\alpha}:=\Prob\left(|\bar{x}-\mu|\le\frac{\alpha s}{\sqrt{N}}\right).
\end{equation}
The parameter $\alpha$ 
is to be interpreted as a ``multiple of the  standard error,''
 and the associated probability $p_{\alpha}$ is one minus the significance level. When $N$ is sufficiently large (e.g., $N\ge30$), by the central limit theorem, the standardized error $\sqrt{N}|\bar{x}-\mu|/s$ approximately follows the standard normal distribution. Hence, $p_{\alpha}$ is approximately the probability of a Gaussian sample whose absolute value is no greater than $\alpha$, i.e.,
\[
p_{\alpha}\approx\erf(\alpha/\sqrt{2}),
\]
where $\erf$ is the error function.

The main result of this section is a probability estimate resembling~\eqref{eqn:conf.intv},
for the sample average $\bar{x}^{(m)}$ \emph{with} numerical error.

\begin{theorem}\label{thm:bound}
Suppose the sample bias is bounded by some $\delta>0$; that is, $|x_i-x_i^{(m)}|\le\delta$ for all $i$, then
\[
\Prob\left\{|\bar{x}^{(m)}-\mu|\le\frac{\alpha}{\sqrt{N}}\left(s^{(m)}+\delta\sqrt{\frac{N}{N-1}}\right)
+\delta\right\}\ge p_{\alpha}.
\]
\end{theorem}

\begin{proof}
  Let                   $x_i-x_i^{(m)}=\delta_i$                   and
  $\bar{x}-\bar{x}^{(m)}=\bar{\delta}$. Form a vector $a$
  with elements $x_i^{(m)}-\bar{x}^{(m)}$ and  another vector $b$ with
  elements  $\delta_i-\bar{\delta}$.  Note   that  $a+b$  consists  of
  elements    $x_i-\bar{x}$.    Then,    the    triangle    inequality
  $\|a+b\|\le\|a\|+\|b\|$ translates to
\[
s\sqrt{N-1} \le s^{(m)}\sqrt{N-1}+\sqrt{\sum_{i=1}^N\left(\delta_i-\bar{\delta}\right)^2}.
\]
Because $|\delta_i|\le\delta$, we have
\[  
\sum_{i=1}^N\left(\delta_i-\bar{\delta}\right)^2
=\sum_{i=1}^N\delta_i^2-N\bar{\delta}^2 
\le   N \delta^2 - N\bar{\delta}^2  \le  N \delta^2 .
\]
Hence,
\[
s\le s^{(m)}+\delta\sqrt{\frac{N}{N-1}}.
\]
Therefore, based on~\eqref{eqn:conf.intv}, we have that with probability greater than $p_{\alpha}$,
\[
|\bar{x}^{(m)}-\mu|\le|\bar{x}-\mu|+|\bar{x}-\bar{x}^{(m)}|
\le\frac{\alpha s}{\sqrt{N}}+|\bar{\delta}|
\le\frac{\alpha}{\sqrt{N}}\left(s^{(m)}+\delta\sqrt{\frac{N}{N-1}}\right)
+\delta,
\]
which concludes the proof.
\end{proof}

Theorem~\ref{thm:bound} gives a computable bound. For any reasonable $\alpha$ (e.g., $3$, which translates to a probability $p_{\alpha}\approx99.73\%$), the error of the sample average $\bar{x}^{(m)}$ is bounded by an expression that involves only the number $N$ of samples, the error tolerance $\delta$, and the standard error $s^{(m)}$. In later experiments, we will use this bound to assess the quality of approximation and show that it is indicative of the true error.

One  may be  interested  in  an appropriate  $\delta$  that makes  the
numerical error comparable with the statistical one. A natural idea is
to let the  tolerance $\delta$ be approximately  the statistical error
bound  $\alpha   s^{(m)}/\sqrt{N}$, or $\beta\alpha   s^{(m)}/\sqrt{N}$ for some small $\beta$ (e.g., $0.1$).  The  following  result   gives  a
straightforward bound  that bypasses the reliance on  $\delta$ for
such a  case; the  overall error  increases to  approximately $(1+\beta)\alpha
s^{(m)}/\sqrt{N}$. This  result may  be used  to adjust  the tolerance
$\delta$ given  the standard  error $s^{(m)}$  obtained in  a previous
calculation.

\begin{corollary}\label{cor:bound}
Under the conditions of Theorem~\ref{thm:bound}, for any $\beta>0$, if $\delta\le\beta\alpha s^{(m)}/\sqrt{N}$, then
\[
\Prob\left\{|\bar{x}^{(m)}-\mu|\le\frac{\alpha s^{(m)}}{\sqrt{N}}\left(1+\beta+\frac{\beta\alpha}{\sqrt{N-1}}\right)\right\}\ge p_{\alpha}.
\]
\end{corollary}

\begin{proof}
Clearly, if $\delta$ satisfies the stated condition, then
\[
\frac{\alpha}{\sqrt{N}}\left(s^{(m)}+\delta\sqrt{\frac{N}{N-1}}\right)+\delta
\le
\frac{\alpha s^{(m)}}{\sqrt{N}}\left(1+\beta+\frac{\beta\alpha}{\sqrt{N-1}}\right).
\]
The proof ends by noting that enlarging an upper bound increases the probability.
\end{proof}

%%%%%%%%%%%%%%%%%%%%%%%%%%%%%%%%%%%%%%%%%%%%%%%%%%%%%%%%%%%%
\section{Error estimation of bilinear form}\label{sec:bilinear}
With the  confidence interval established  in Theorem~\ref{thm:bound},
we  now  consider how  to  reliably  estimate the  difference  between
$u^Tf(A)u$  and $\|u\|^2e_1^Tf(T_m)e_1$,  because this  error will  be
used  to  check against  the  tolerance  $\delta$ for  convergence.  A
challenge in computing  matrix functions based on  the Lanczos method
is that  good error estimates are  hardly known, except for  the simple
case  of   the  exponential~\cite{Saad1992}   because  of   its  fast
convergence.         Several         error         bounds         were
proposed~\cite{Frommer2009,Frommer2013} but they are generally pessimistic and may deviate from the
true error  by one or  more order  of magnitudes, when  convergence is
slow. In this section, we propose  a general technique applicable to a
wide variety of functions and also to ill conditioned matrices.

%%%%%%%%%%%%%%%%%%%%%%%%%%%%%%%%%%%%%%%%%%%%%%%%%%%%%%%%%%%%
\subsection{Incremental and cumulative error}\label{sec:three.err}
Omitting the common, known factor $\|u\|^2$, we define the quantity of interest
\[
\rho_m:=v_1^Tf(A)v_1-e_1^Tf(T_m)e_1,
\]
where recall that $v_1=u/\|u\|$.
If the Lanczos iteration~\eqref{eqn:lanczos} is run to the end, we have%
\footnote{In the case of breakdown, restart with a new vector orthogonal to the previous Krylov subspace(s). Hence, \eqref{eqn:AVVT.2} always holds, with some $\beta_k$('s) possibly being zero.}
\begin{equation}\label{eqn:AVVT.2}
AV_n=V_nT_n,
\end{equation}
where the subscript $m$ in~\eqref{eqn:AVVT} is replaced by the matrix dimension $n$ and the remainder term vanishes. Therefore, $\rho_m$ is nothing but $e_1^Tf(T_n)e_1-e_1^Tf(T_m)e_1$.

We call $\rho_m$ the \emph{bilinear form error}. In order to quantify this error, we define two additional terms closely related to $\rho_m$:
\begin{enumerate}
\item \emph{cumulative error}: $d_{m,m'}:=e_1^Tf(T_{m'})e_1-e_1^Tf(T_m)e_1$ for $m'>m$;
\item \emph{incremental error}: $d_m:=d_{m,m+1}$.
\end{enumerate}
Clearly, the incremental error accounts for one step of the difference and the cumulative error accumulates the incremental errors for $m'-m$ steps. In other words,
\[
d_{m,m'}=\sum_{i=m}^{m'-1}d_i
\quad\text{for all $m'>m$ and particularly,}\quad
\rho_m=d_{m,n}=\sum_{i=m}^{n-1}d_i.
\]

%%%%%%%%%%%%%%%%%%%%%%%%%%%%%%%%%%%%%%%%%%%%%%%%%%%%%%%%%%%%
\subsection{Rational approximation}\label{sec:rat.approx}
To estimate the bilinear form error $\rho_m$, we begin with the incremental error $d_m$, because it can be computed economically without evaluating $f(T_m)$ for every $m$. The idea is to express $f$ with the Cauchy integral
\[
f(a)=\frac{-1}{2\pi i}\int_{\Gamma} f(z)(a-z)^{-1}\,dz,
\]
where $\Gamma$, inside which $f$ is analytic, is a contour enclosing $a$. Let the contour integral be approximated by using a quadrature rule
\[
f(a)\approx\frac{-1}{2\pi i}\sum_{k=1}^K w_kf(z_k)(a-z_k)^{-1},
\]
where $z_k$ and $w_k$ are the quadrature points and weights, respectively. Then, we effectively obtain a rational approximation of $f(a)$:
\begin{equation}\label{eqn:rK}
r_K(a)=\sum_{k=1}^Kc_k(a-z_k)^{-1},
\quad\text{with}\quad
c_k=\frac{-w_kf(z_k)}{2\pi i},
\end{equation}
where the poles $z_k$ are the same as the quadrature points and the coefficients $c_k$ are related to the quadrature weights $w_k$.

This  intricate relationship  between a  rational approximation  and a
contour integral approximated by quadrature  is well known. It
is a  valuable device  for computing  a function of  a matrix  times a
vector, $f(A)b$,  to high  accuracy, because  of the much faster
convergence of  rational approximations compared with polynomial approximations, provided that  shifted linear
systems  with  respect  to  $A$   are  solved  in  a  backward  stable
manner~\cite{Hale2008}. A  challenge for  applying this idea  to large
$A$ in  practice, is that  solving the systems  by using a  direct and
stable method might  not always be a viable option.  Here, we will not
discuss in detail  the pros and cons of various  methods for computing
$f(A)b$, because  the comparison  is irrelevant.  Instead, we  use the
device         as         a         tool         for         analyzing
$d_m=e_1^Tf(T_{m+1})e_1-e_1^Tf(T_m)e_1$.  An appealing  consequence of
the  \emph{very} fast  convergence is  that the  number of  quadrature
points, $K$, need not be large to get sufficiently good approximations.

From a practical stand point, we refer the readers to articles~\cite{Trefethen2006,Hale2008} and references therein for the rational approximations of a wide variety of functions used in applications (including, e.g., the exponential, the logarithm, and the square root). Some approximations are not written in the canonical form~\eqref{eqn:rK}, but we will explain the simple modifications in later experiments. Moreover, because of conjugacy, and because we are interested in real arguments only, the number of summation terms in~\eqref{eqn:rK} may often be reduced by half. Therefore, throughout the paper we assume the following rational approximation:
\begin{equation}\label{eqn:rK2}
r_K(x):=\Re\left\{\sum_{k=1}^Kc_k(x-z_k)^{-1}\right\}, \quad x\in\real.
\end{equation}

Because the spectrum interval of $T_m$ always stays inside that of $A$
(owing to  the interlacing eigenvalue  theorem), it suffices to  use a
contour  that encloses  the spectrum  interval  of $A$  so that  the
incremental  error  $d_m$  is   well  approximated  by  the  following
quantity:
\begin{align}
d_m^K &:=e_1^Tr_K(T_{m+1})e_1-e_1^Tr_K(T_m)e_1 \nonumber\\
&=\Re\left\{\sum_{k=1}^K c_k[e_1^T(T_{m+1}-z_kI)^{-1}e_1-e_1^T(T_m-z_kI)^{-1}e_1]\right\}. \label{eqn:dmK}
\end{align}
The following result is in preparation for an iterative algorithm that
efficiently computes $d_m^K$.

\begin{proposition}\label{prop:dmK}
We have,
\begin{equation}\label{eqn:dmK3}
d_m^K=-\Re\left\{\sum_{k=1}^Kc_k\beta_{m+1}[e_{m+1}^T(T_{m+1}-z_kI)^{-1}e_1][e_m^T(T_m-z_kI)^{-1}e_1]\right\},
\end{equation}
where recall that $\beta_{m+1}$ is the last of the coefficients $\beta_2$, $\beta_3$, \ldots in the Lanczos iteration~\eqref{eqn:lanczos}.
\end{proposition}

\begin{proof}
Note that $T_m-z_kI$ is a top-left block of $T_{m+1}-z_kI$. Thus, we are seeking the difference between the $(1,1)$ element of the inverse of a matrix and that of its top-left block. Recall the following identity:
\[
\begin{bmatrix}R & S \\ T & U\end{bmatrix}^{-1}
=\begin{bmatrix}R^{-1}+R^{-1}SQTR^{-1} & -R^{-1}SQ \\ -QTR^{-1} & Q\end{bmatrix},
\quad\text{with}\quad
Q=(U-TR^{-1}S)^{-1}.
\]
It indicates that the difference between the $(1,1)$ element of $\left[\begin{smallmatrix}R & S \\ T & U\end{smallmatrix}\right]^{-1}$ and that of $R^{-1}$ is $e_1^TR^{-1}SQTR^{-1}e_1$.

In our setting, $\left[\begin{smallmatrix}R & S \\ T & U\end{smallmatrix}\right]=T_{m+1}-z_kI$ and $R=T_m-z_kI$. Therefore, $S=\beta_{m+1}e_me_1^T$ and thus
\[
e_1^TR^{-1}SQTR^{-1}e_1=\beta_{m+1}(e_1^TR^{-1}e_m)(e_1^TQTR^{-1}e_1).
\]
Clearly, $e_1^TR^{-1}e_m=e_1^T(T_m-z_kI)^{-1}e_m=e_m^T(T_m-z_kI)^{-1}e_1$ and $e_1^TQTR^{-1}e_1=-e_{m+1}^T(T_{m+1}-z_kI)^{-1}e_1$. Hence,
\[
e_1^T(T_{m+1}-z_kI)^{-1}e_1-e_1^T(T_m-z_kI)^{-1}e_1=\\
-\beta_{m+1}[e_{m+1}^T(T_{m+1}-z_kI)^{-1}e_1][e_m^T(T_m-z_kI)^{-1}e_1],
\]
which concludes the proof.
\end{proof}

For  conciseness, in  what  follows, the  three  errors introduced  in
Section~\ref{sec:three.err}  may mean  either  the originally  defined
terms,     or    the     approximated    terms     through    rational
approximation~\eqref{eqn:rK2}. This  abuse of language will  not cause
confusion in the current context. The approximated terms have a superscript $K$
attached to the notation, just like $d_m^K$.
The following result states that the error in the approximated terms is always bounded by two times the uniform error between $f$ and $r_K$.

\begin{theorem}\label{thm:rational.err}
Let $\|f-r_K\|_{\infty}:=\max_{x\in[\lambda_{\min},\lambda_{\max}]}|f(x)-r_K(x)|=\epsilon$, where $\lambda_{\min}$ and $\lambda_{\max}$ are the smallest and largest eigenvalues of $A$, respectively. For any $m$ and $m'>m$ \textup{(}including the case of incremental error $m'=m+1$ and bilinear form error $m'=n$\textup{)}, the cumulative error admits
\[
|d_{m,m'}-d_{m,m'}^K|\le2\epsilon.
\]
\end{theorem}

\begin{proof}
Let the spectral decomposition of $T_m$ be $S\Theta S^T$. Then
\[
|e_1^Tf(T_m)e_1-e_1^Tr_K(T_m)e_1|=|e_1^TS[f(\Theta)-r_K(\Theta)]S^Te_1|
\le\|f-r_K\|_{\infty}=\epsilon,
\]
where the inequality comes from the fact that the vector $S^Te_1$ has a unit 2-norm. Since this inequality holds for all $m$, we have
\begin{align*}
|d_{m,m'}-d_{m,m'}^K|
&=|[e_1^Tf(T_{m'})e_1-e_1^Tf(T_m)e_1]-[e_1^Tr_K(T_{m'})e_1-e_1^Tr_K(T_m)e_1]|\\
&\le|e_1^Tf(T_{m'})e_1-e_1^Tr_K(T_{m'})e_1|+|e_1^Tf(T_m)e_1-e_1^Tr_K(T_m)e_1|
\le2\epsilon,
\end{align*}
which concludes the proof.
\end{proof}

%%%%%%%%%%%%%%%%%%%%%%%%%%%%%%%%%%%%%%%%%%%%%%%%%%%%%%%%%%%%
\subsection{Iterative algorithm for computing the incremental error}
\label{sec:incre.err}
With Proposition~\ref{prop:dmK}, if we define
\begin{equation}\label{eqn:eta.m.k}
\eta_m^k:=e_m^T(T_m-z_kI)^{-1}e_1,
\end{equation}
then the incremental error~\eqref{eqn:dmK3} is simplified as
\begin{equation}\label{eqn:dmK2}
d_m^K=-\Re\left\{\sum_{k=1}^Kc_k\beta_{m+1}\eta_{m+1}^k\eta_m^k\right\}.
\end{equation}
Hence, an efficient computation of $d_m^K$ comes from an iterative technique that economically computes $\eta_{m+1}^k$ based on $\eta_m^k$.

For convenience,  we temporarily  omit the  index $k$
that distinguishes  between different poles.  They will return  at the
end of  this subsection.   We seek an  inexpensive update  formula for
$\eta_{m+1}$   based  on   $\eta_m$.   Assume   an  LU   factorization
$T_m-zI=L_mU_m$ and  let $u_m$ be  the bottom-right corner  element of
$U_m$. Then for one additional step, we have
\[
T_{m+1}-zI
=\begin{bmatrix}T_m-zI & \beta_{m+1}e_m \\ \beta_{m+1}e_m^T & \alpha_{m+1}-z\end{bmatrix}
=\begin{bmatrix}L_m & \\ (\beta_{m+1}/u_m)e_m^T & 1\end{bmatrix}
\begin{bmatrix}U_m & \beta_{m+1}e_m \\ & u_{m+1}\end{bmatrix},
\]
where
\begin{equation}\label{eqn:um1}
u_{m+1}=\alpha_{m+1}-z-\beta_{m+1}^2/u_m.
\end{equation}
Because
\[
\eta_m=e_m^T(T_m-zI)^{-1}e_1=e_m^TU_m^{-1}L_m^{-1}e_1=u_m^{-1}e_m^TL_m^{-1}e_1,
\]
we have $\eta_{m+1}=u_{m+1}^{-1}e_{m+1}^TL_{m+1}^{-1}e_1$. Thus, by noting the equality
\[
\begin{bmatrix}L_m & \\ (\beta_{m+1}/u_m)e_m^T & 1\end{bmatrix}
\begin{bmatrix}L_m^{-1}e_1 \\ -(\beta_{m+1}/u_m)e_m^TL_m^{-1}e_1\end{bmatrix}=
\begin{bmatrix}e_1 \\ 0\end{bmatrix},
\]
we obtain
\begin{equation}\label{eqn:etam1}
\eta_{m+1}=-u_{m+1}^{-1}(\beta_{m+1}/u_m)e_m^TL_m^{-1}e_1
=-\beta_{m+1}\eta_m/u_{m+1}.
\end{equation}

Hence, to efficiently compute  the quantity $\eta_m$, it suffices to  insert a few
lines related to~\eqref{eqn:um1} and~\eqref{eqn:etam1}  into the
existing Lanczos  iteration~\eqref{eqn:lanczos}. Then,  with $\eta_m$,
the    incremental   error    is    computed in a    straightforward manner    by
using~\eqref{eqn:dmK2}. We  now put back  the index $k$  and summarize
this computation in Algorithm~\ref{algo:add.simple}. Note that because
at  the $m$-th  Lanczos step,  only  $\eta_m^k$ is  available but  not
$\eta_{m+1}^k$, we  need to shift  the index $m$  by $1$.

{\renewcommand{\baselinestretch}{1.2}
\begin{algorithm}[ht]
\caption{Computing $d_m^K$ for $m=1,2,\ldots$}
\label{algo:add.simple}
\begin{algorithmic}[1]
\For{$m=1,2,\ldots$}
\State Run one step of Lanczos, producing $\alpha_m$ and $\beta_{m+1}$; see~\eqref{eqn:lanczos}
\State Initialize $d_{m-1}^K\gets0$ if $m\ne1$
\label{algo.ln:add.simple.start}
\For{$k=1,\ldots,K$}
\State \textbf{if} $m=1$ \textbf{then} compute
$u_1^k=\alpha_1-z_k$ and
$\eta_1^k=(u_1^k)^{-1}$
\State \textbf{else} compute
$u_{m}^k=\alpha_{m}-z_k-\beta_{m}^2/u_{m-1}^k$,\,\,
$\eta_{m}^k=-\beta_{m}\eta_{m-1}^k/u_{m}^k$, and
\State \hspace{0.6cm} update $d_{m-1}^K\gets d_{m-1}^K-\Re\{c_k\beta_m\eta_m^k\eta_{m-1}^k\}$
\EndFor \label{algo.ln:add.simple.end}
\EndFor
\end{algorithmic}
\end{algorithm}}

The cost of computing $d_m^K$ in this manner for each $m$ is simply
$O(K)$. It is trivial compared with that of the Lanczos iteration, as
long as the number of quadrature points, $K$, is far smaller than the
matrix dimension $n$. This is because computing the $\alpha_m$'s and
$\beta_{m+1}$'s requires matrix-vector multiplications and vector
inner products, which have an $O(n+\nnz)$ cost, where $\nnz$ denotes
the number of nonzeros of the matrix. This approach is also more
economical than computing~\eqref{eqn:eta.m.k} directly through
factorization for every $m$, because the factorization/solve admits an
$O(mK)$ cost.

We note that the algorithm is equivalent to Theorem 3.9 of Golub and Meurant~\cite{Golub2009}, derived from a different angle.

%%%%%%%%%%%%%%%%%%%%%%%%%%%%%%%%%%%%%%%%%%%%%%%%%%%%%%%%%%%%
\subsection{Estimating the bilinear form error}
We have presented an iterative algorithm for computing the incremental error $d_m^K$ in the preceding subsection. Due to the rational approximation, the cumulative error is now denoted by
\[
d_{m,m'}^K=\sum_{i=m}^{m'-1}d_m^K.
\]
When the accumulation is done to the end (i.e., $m'=n$), we reach the bilinear form error $\rho_m^K=d_{m,n}^K$.

It is,  of course, impractical  to accumulate incremental  errors till
$m'=n$, because  this requires  running the  Lanczos algorithm  to the
end. If  $e_1^Tf(T_m)e_1$ converges reasonably fast,  one expects that
$m'$ need  not be much larger than  $m$ for the cumulative error to be
nearly the bilinear form error. In fact, for the exponential function,
the extremely  fast convergence  indicates that the  incremental error
alone,  without accumulation,  is  already sufficiently  close to  the
bilinear form error  (see, e.g., Figure~\ref{fig:bilinear.form.err}(a)
in  Section~\ref{sec:exp.lap2d.bilinear.err}).  For  other  functions,
then,  one needs  a strategy  to find  a suitable  $m'$ such  that the
cumulative error is a good estimate of the bilinear form error.

This task is challenging because the incremental error is difficult to characterize. We therefore apply some simplified model that simulates the behavior of the sequence of incremental errors $d_1^K$, $d_2^K$, $d_3^K$ \ldots. The following known facts motivate a geometric progression model:

\begin{enumerate}
\item If $f^{(2m)}>0$ in the spectrum interval of $A$ for all $m$, then $e_1^Tf(T_m)e_1<e_1^Tf(T_{m+1})e_1$; and a similar statement holds when both inequalities change direction~\cite{Lagomasino2008,Golub2009}. Such a monotone convergence comes from the fact that the bilinear form error $e_1^Tf(T_n)e_1-e_1^Tf(T_m)e_1$ is $f^{(2m)}(x')$ times a positive factor, for some $x'$ inside the spectrum interval, based on a standard argument of Gauss quadratures. Many functions in applications possess this property, including $\exp(x)$, $\log(x)$, $\Gamma(x)$, $\tanh(\sqrt{x})$, and $x^{\alpha}$ for $\alpha<1$. A consequence is that under this condition, the incremental errors $d_m=e_1^Tf(T_{m+1})e_1-e_1^Tf(T_m)e_1$ have a constant sign.

\item The bilinear form error converges exponentially (i.e., $|\rho_m|=O(c^{-2m})$ for some $c>1$) if $f$ is analytic in the spectrum interval and analytically continuable in an open Bernstein ellipse whose foci are the two ends of the interval~\cite{Trefethen2012}. The exponential convergence is, again, owing to a standard property of Gauss quadratures. Hence, if $\rho_m$ is precisely $c^{-2m}$ up to a constant multiplicative factor, we have, for the incremental errors, $d_{m+1}/d_m=c^{-2}$ for all $m$.
\end{enumerate}

Based on these facts, we will use a geometric progression to
approximately model the behavior of the sequence of incremental
errors. The following result is a basis of the strategy we propose for
finding an appropriate $m'$ such that the cumulative error
$d_{m,m'}^K$ is a good approximation to the bilinear form error
$\rho_m^K$.

\begin{proposition}\label{prop:geom}
Let $\{a_i\}_{i=1}^{n-1}$ be a positive and decreasing geometric progression; that is, $a_{i+1}/a_i$ is a positive constant $<1$ for all $i$. The sequence may be infinite, in which case $n=\infty$. Given $m$, $m'>m$, and $t<1$ such that $a_{m'}/a_m\le t$, we have
\[
\frac{\sum_{i=m'}^{n-1}a_i}{\sum_{i=m}^{m'-1}a_i}\le
\begin{cases}
t/(1-t), & \text{if } n-m'>m'-m,\\
t, & \text{otherwise}.
\end{cases}
\]
\end{proposition}

\begin{proof}
Denote by $c=a_{i+1}/a_i$ the progression ratio. Clearly,
\[
\frac{\sum_{i=m'}^{n-1}a_i}{\sum_{i=m}^{m'-1}a_i}=\frac{c^{m'-m}-c^{n-m}}{1-c^{m'-m}}.
\]
Because the right-hand side of the above equality is an increasing function for $0<c<1$ and because $c^{m'-m}=a_{m'}/a_m\le t$, we obtain
\[
\frac{\sum_{i=m'}^{n-1}a_i}{\sum_{i=m}^{m'-1}a_i}\le\frac{t}{1-t}\left(1-t^{\frac{n-m'}{m'-m}}\right).
\]
We conclude the proof by noting that if $n-m'>m'-m$, then $1-t^{(n-m')/(m'-m)}\le1$; otherwise, $1-t^{(n-m')/(m'-m)}\le 1-t$.
\end{proof}

Proposition~\ref{prop:geom}  says  that  if $\{a_i\}$  is  a  positive
sequence with elements  decreasing at the same rate, and  if we pick a
pair  of indices  $m$ and  $m'$ such  that the  ratio $a_{m'}/a_m$  is
bounded by  some value  $t<1$, then the ratio  between the  summation from
$a_{m'}$  to  the  end  of  the  sequence,  and  that  from  $a_m$  to
$a_{m'-1}$, is  also bounded.  The bound, regardless  of how  long the
sequence is, can be  made small by using a small  $t$. For example, if
$t=0.1$, then  the bound,  either $t/(1-t)$  or $t$,  is approximately
$0.1$.

If the  sequence of incremental  errors $d_m^K$ follows  precisely the
geometric   progression  of   the  proposition,   then  applying   the
proposition we  see that  the ratio between  $\rho_m^K-d_{m,m'}^K$ and
$d_{m,m'}^K$ is bounded  by approximately $0.1$, when using $t=0.1$. In
other  words,  the  cumulative  error $d_{m,m'}^K$  is  close  to  the
bilinear form error $\rho_m^K$. This  closeness is sufficient for an error
estimation,   because  if   $\epsilon$   is  the   tolerance  and   if
$d_{m,m'}^K=\epsilon$, then the bilinear form error $\rho_m^K$ will be
at most approximately $1.1$ times of $\epsilon$.

In addition, if the incremental errors are negative but their absolute values follow a
geometric progression, we may clearly draw the same ``sufficient closeness'' 
conclusion by using an analogous argument.

Hence, to summarize, the strategy to estimate the bilinear form error $\rho_m^K$ at the $m$-th Lanczos step, is to find the smallest $m'>m$ such that $|d_{m'}^K|/|d_m^K|\le t$ and use $d_{m,m'}^K$ as an approximation of $\rho_m^K$. For all practical purposes, it suffices to fix the threshold $t$ to be $0.1$.

%%%%%%%%%%%%%%%%%%%%%%%%%%%%%%%%%%%%%%%%%%%%%%%%%%%%%%%%%%%%
\subsection{Analysis}
The strategy proposed in the preceding subsection is motivated by a geometric progression model of the bilinear form errors. In practice, the errors rarely follow such a pattern exactly. In particular, although asymptotically the errors behave like a geometric progression due to the exponential convergence, they exhibit much variety before entering the asymptotic regime.

In this subsection, we analyze two example scenarios to gain a better understanding of the effectiveness of the error estimate. These scenarios are pictorially illustrated in Figure~\ref{fig:bilinear.error.pic}, where the left plot indicates that the logarithmic error decreases slowly initially, and the right plot suggests otherwise. In what follows, we give results analogous to Proposition~\ref{prop:geom}, one for each scenario.

\begin{figure}[ht]
\centering
\subfigure[Slow decrease initially]{
  \includegraphics[width=.4\linewidth]{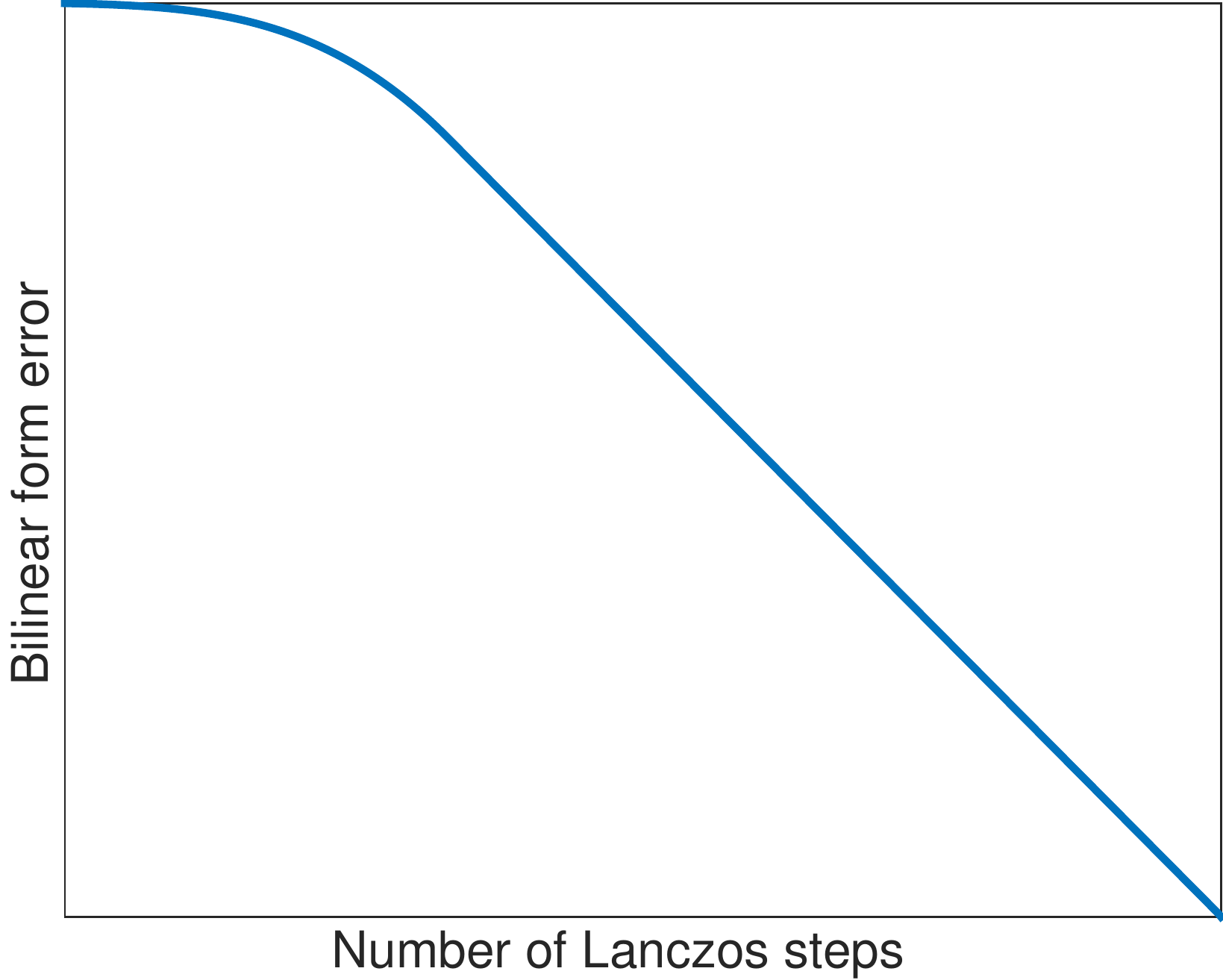}}\hspace{.5cm}
\subfigure[Fast decrease initially]{
  \includegraphics[width=.4\linewidth]{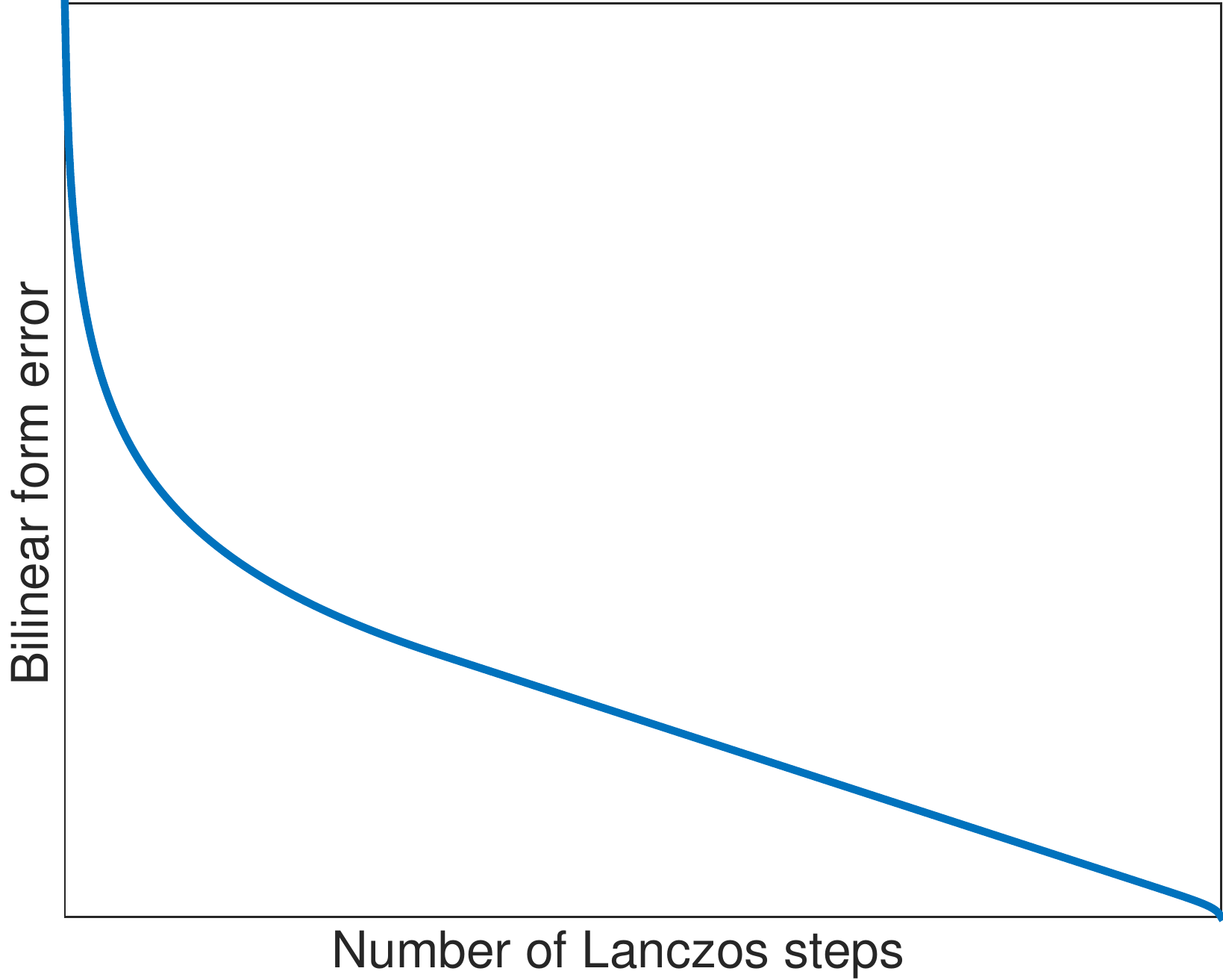}}
\caption{Pictorial illustration of the logarithmic bilinear form error over Lanczos iterations.}
\label{fig:bilinear.error.pic}
\end{figure}

An example of the scenario illustrated in Figure~\ref{fig:bilinear.error.pic}(a) is that the incremental errors $d_m^K$ admit a decreasing ratio $d_2^K/d_1^K \ge d_3^K/d_2^K \ge \cdots$. One may show simply through induction that in such a case, the bilinear form errors $\rho_m^K$ also admit a decreasing ratio $\rho_2^K/\rho_1^K \ge \rho_3^K/\rho_2^K \ge \cdots$, which gives a concave shape of the error curve. Then, we obtain the same bound as that in Proposition~\ref{prop:geom}, which is a special case of the following result.

\begin{proposition}\label{prop:geom2}
Let $\{a_i\}_{i=1}^{n-1}$ be a positive and nonincreasing sequence and let $\{a_{i+1}/a_i\}_{i=1}^{n-2}$ be nonincreasing as well. Given $m$, $m'>m$, and $t<1$ such that $a_{m'}/a_m\le t$, we have
\[
\frac{\sum_{i=m'}^{n-1}a_i}{\sum_{i=m}^{m'-1}a_i}\le
\begin{cases}
t/(1-t), & \text{if } n-m'>m'-m,\\
t, & \text{otherwise}.
\end{cases}
\]
\end{proposition}

\begin{proof}
Let $a_{m'}/a_{m'-1}=c$. Then, because $a_{i+1}/a_i$ is nonincreasing for all $i$, we have
\[
a_{m'}+a_{m'+1}+\cdots+a_{n-1} \le a_{m'}(1+c+c^2+\cdots+c^{n-m'-1})
\]
and
\[
a_m+a_{m+1}+a_{m'-1} \ge a_m(1+c+c^2+\cdots+c^{m'-m-1}).
\]
Therefore,
\[
\frac{\sum_{i=m'}^{n-1}a_i}{\sum_{i=m}^{m'-1}a_i} \le \frac{a_{m'}(1-c^{n-m'})}{a_m(1-c^{m'-m})} \le t\,\,\frac{1-c^{n-m'}}{1-c^{m'-m}}.
\]
Furthermore, from
\[
t\ge\frac{a_{m'}}{a_m}=\frac{a_{m+1}}{a_m}\frac{a_{m+2}}{a_{m+1}}\cdots\frac{a_{m'}}{a_{m'-1}}
\ge\left(\frac{a_{m'}}{a_{m'-1}}\right)^{m'-m}
=c^{m'-m},
\]
we conclude the proof by noting that if $n-m'>m'-m$, then $(1-c^{n-m'})/(1-c^{m'-m})\le1/(1-c^{m'-m})\le1/(1-t)$; otherwise, $(1-c^{n-m'})/(1-c^{m'-m})\le1$.
\end{proof}

On the other hand, an example of the scenario illustrated in Figure~\ref{fig:bilinear.error.pic}(b) is that the incremental errors $d_m^K$ admit an increasing ratio $d_2^K/d_1^K \le d_3^K/d_2^K \le \cdots$. The worst case is that there exist two consecutive integers $m$ and $m'=m+1$ such that the ratio $d_{m+1}^K/d_m^K$ is quite small (e.g., lower than the threshold $t$), but the incremental errors afterward decay too slowly, such that the cumulative error $d_{m,m'}^K$ constitutes only a tiny portion of the overall error $\rho_m^K$. Hence, we consider a case where the incremental errors cannot abruptly change. In particular, let us assume that the beginning of the sequence $\{d_m^K\}$ is proportional to $m^{-(p+1)}$ for some $p>0$. At a certain point (when $m=s$), the sequence decreases at a constant rate, which results in an exponential decay pattern of the bilinear form error. The rate $c$ is equal to $(s-1)^{p+1}/s^{p+1}$ such that the transition of the decaying patterns is smooth. For this scenario, we have the following result.

\begin{proposition}\label{prop:geom3}
Let $\{a_i\}_{i=m}^{n-1}$ be a sequence
\[
\frac{1}{m^{p+1}}, \,\, \frac{1}{(m+1)^{p+1}}, \,\, \ldots, \frac{1}{(s-1)^{p+1}},
\,\, \frac{1}{s^{p+1}}, \,\, \frac{c}{s^{p+1}}, \,\,\ldots,\frac{c^{n-s-1}}{s^{p+1}}
\]
for some integer $s\in(m,n)$ and real number $p>0$, where $c=(s-1)^{p+1}/s^{p+1}$. Given integer $m'\in(m,s)$ and $t<1$ such that $a_{m'}/a_m\le t$, we have
\[
\frac{\sum_{i=m'}^{n-1}a_i}{\sum_{i=m}^{m'-1}a_i}\le
\dfrac{1+\frac{p}{m'}-\left(\frac{m'}{s}\right)^p\left[1+\frac{p}{s}-\frac{p}{p+1}\frac{1}{1-p/(2s)}\right]}{t^{-\frac{p}{p+1}}-1}.
\]
\end{proposition}

\begin{proof}
Because $1/x^{p+1}$ is monotonically decreasing for $x>0$, from the definition of integral (area under curve), we have for any positive integers $i$ and $j$ where $j>i$,
\[
\frac{1}{i^{p+1}}+\cdots+\frac{1}{(j-1)^{p+1}}\ge\int_i^j\frac{dx}{x^{p+1}}
=\frac{1}{p}\left(\frac{1}{i^p}-\frac{1}{j^p}\right),
\]
and
\[
\frac{1}{i^{p+1}}+\cdots+\frac{1}{(j-1)^{p+1}}\le\int_i^j\frac{dx}{x^{p+1}}
+\left(\frac{1}{i^{p+1}}-\frac{1}{j^{p+1}}\right)
=\frac{1}{p}\left(\frac{1}{i^p}-\frac{1}{j^p}\right)
+\left(\frac{1}{i^{p+1}}-\frac{1}{j^{p+1}}\right).
\]
Therefore,
\begin{align*}
a_{m'}+\cdots+a_{n-1}
&=\frac{1}{(m')^{p+1}}+\cdots+\frac{1}{(s-1)^{p+1}}
+\frac{1}{s^{p+1}}+\cdots+\frac{c^{n-s-1}}{s^{p+1}}\\
&\le\frac{1}{p}\left(\frac{1}{(m')^p}-\frac{1}{s^p}\right)
+\left(\frac{1}{(m')^{p+1}}-\frac{1}{s^{p+1}}\right)
+\frac{1}{s^{p+1}}\frac{1-c^{n-s}}{1-c},
\end{align*}
and
\[
a_m+\cdots+a_{m'-1}\ge\frac{1}{p}\left(\frac{1}{m^p}-\frac{1}{(m')^p}\right).
\]
Hence, the ratio
\begin{equation}\label{eqn:r1}
\frac{\sum_{i=m'}^{n-1}a_i}{\sum_{i=m}^{m'-1}a_i}\le
\dfrac{1-\left(\frac{m'}{s}\right)^p+\frac{p}{m'}-\left(\frac{m'}{s}\right)^p\frac{p}{s}+\left(\frac{m'}{s}\right)^p\frac{p}{s}\frac{1-c^{n-s}}{1-c}}{\left(\frac{m'}{m}\right)^p-1}.
\end{equation}
Moreover, from $a_{m'}/a_m\le t$ we have
\begin{equation}\label{eqn:r2}
\left(\frac{m'}{m}\right)^p
=\left[\left(\frac{a_m}{a_{m'}}\right)^{-\frac{1}{p+1}}\right]^p
\ge t^{-\frac{p}{p+1}},
\end{equation}
and from $c=(s-1)^{p+1}/s^{p+1}$ we have
\begin{equation}\label{eqn:r3}
1-c=1-\left(1-\frac{1}{s}\right)^{p+1}
\ge1-\left[1-\frac{p+1}{s}+\frac{(p+1)p}{2s^2}\right]
=\frac{p+1}{s}\left(1-\frac{p}{2s}\right).
\end{equation}
Then, substituting~\eqref{eqn:r2} and~\eqref{eqn:r3}, together with the fact that $1-c^{n-s}<1$, into~\eqref{eqn:r1}, we reach the inequality result of the proposition.
\end{proof}

The bound in Proposition~\ref{prop:geom3} is slightly more obscure than that of Proposition~\ref{prop:geom2}, but it offers a qualitative interpretation. When $p\ll m'$ and $s$, the terms $p/m'$ and $p/s$ in the numerator are nearly zero, and hence the bound reads
\begin{equation}\label{eqn:bound}
\frac{\sum_{i=m'}^{n-1}a_i}{\sum_{i=m}^{m'-1}a_i}\lessapprox
\dfrac{1-\frac{1}{p+1}\left(\frac{m'}{s}\right)^p}{t^{-\frac{p}{p+1}}-1}.
\end{equation}
When $p\ge1$, we could even enlarge the right-hand side by omitting the term $\frac{1}{p+1}\left(\frac{m'}{s}\right)^p$, which results in a bound $\left(t^{-\frac{p}{p+1}}-1\right)^{-1}$. If $t=0.1$, this bound $\le0.47$, sufficient for error estimation. When $p<1$, the term $\frac{1}{p+1}\left(\frac{m'}{s}\right)^p$ may be nonnegligible, especially when $m'$ is not too far from $s$. This term offsets the possibly small value of $t^{-\frac{p}{p+1}}-1$. The net result is that the bound~\eqref{eqn:bound} is not too large. For example, if $m'/s=0.5$, then the bound $\le0.74$, again sufficient for error estimation.

%%%%%%%%%%%%%%%%%%%%%%%%%%%%%%%%%%%%%%%%%%%%%%%%%%%%%%%%%%%%
\section{Overall algorithm and parameter setting}
With the developments in the  preceding sections, we now summarize the
overall  algorithm that  includes both  approximating $\tr(f(A))$  and
estimating   the   approximation   error.   Details   are   shown   in
Algorithm~\ref{algo:hutchinson1}.

%%%%%%%%%%%%%%%%%%%%%%%%%%%%%%%%%%%%%%%%%%%%%%%%%%%%%%%%%%%%
\subsection{Algorithm}
The  procedure  begins with  approximating  $\tr(f(A))$  by using  $N$
independent  and unbiased  samples $u_i^Tf(A)u_i$.  Each sample  is in
turn approximated  by $\|u_i\|^2e_1^Tf(T_m)e_1$  based on  the Lanczos
method. Here, $m$ is the number  of Lanczos steps and it is implicitly
determined  by  an error  tolerance  $\delta$  that ensures  that  the
difference  between  $u_i^Tf(A)u_i$ and  $\|u_i\|^2e_1^Tf(T_m)e_1$  is
bounded by the tolerance. To estimate the difference between these two
quantities (termed ``bilinear form error'')  at each Lanczos step $m$,
an incremental error $d_m^K$ is  computed based on a simple recurrence
summarized  in Algorithm~\ref{algo:add.simple}.  Then, an  estimate of
the  bilinear form  error,  termed ``cumulative  error,'' is  computed
based on the incremental errors.

To  be specific,  we need  to trace  back a  few steps  (say, at  step
$\ud{m}<m$) to obtain  an accurate approximation of  the bilinear form
error. Therefore,  the algorithmic  progression is  opposite to  how the
cumulative error  is defined  in the preceding  section (where  we used the
notation  $m<m'$ and thought forwardly).  Algorithmically, we  say that
$d_{\ud{m},m}^K$  is an  accumulation of  the incremental  errors from
$d_{\ud{m}}^K$ to $d_{m-1}^K$. Hence, whenever a new incremental error
is obtained in  a certain Lanczos step, it is  added to the cumulative
errors      for       all      the      previous       steps      (see
line~\ref{algo.ln:hutchinson1.cum}                                  of
Algorithm~\ref{algo:hutchinson1}).    We    maintain    a    threshold
$t=0.1$.  For every  $m$, there is an associated  $\ud{m}$  that is  the
smallest integer satisfying $|d_m^K|/|d_{\ud{m}}^K|\ge  t$. If at some
step  $m$   with  the   associated  $\ud{m}$,  the   cumulative  error
$d_{\ud{m},m}^K$ falls below  the scaled tolerance $\delta/\|u_i\|^2$,
then   we    consider   that    the   bilinear    form   approximation
$\|u_i\|^2e_1^Tf(T_{\ud{m}})e_1$   has  converged   to  $u_i^Tf(A)u_i$
within  a  tolerance  $\delta$.  This  concludes  the  computation  of
$u_i^Tf(A)u_i$.

%Note that for a reason that was made clear in the preceding
%section, the additional cost for monitoring the error is trivial
%compared with the overall trace computation.

\begin{algorithm}[ht]
\caption{Estimating $\tr(f(A))$ with confidence interval}
\label{algo:hutchinson1}
\begin{algorithmic}[1]
\Require Matrix $A$, function $f$, number $N$ of random vectors, Lanczos tolerance $\delta$, threshold $t=0.1$
\State If needed by line~\ref{algo.ln:hutchinson1.est.spectrum}, estimate the spectrum interval of $A$
\State Obtain a rational approximation of $f$ (in the spectrum interval)
\label{algo.ln:hutchinson1.est.spectrum}
\Loop{ $i=1,2,\ldots,N$}
\State Generate a random vector $u_i$ and use $u_i/\|u_i\|$ as the starting vector of Lanczos
\Loop{ $m=1,2,\ldots$}
\State Run one step of Lanczos; perform reorthogonalization if necessary
\State Compute incremental error $d_m^K$ (i.e., lines~\ref{algo.ln:add.simple.start}--\ref{algo.ln:add.simple.end} of Algorithm~\ref{algo:add.simple})
\State Update cumulative error $d_{\ud{m},m}^K\gets d_{\ud{m},m}^K+d_m^K$ for all $\ud{m}<m$ \label{algo.ln:hutchinson1.cum}
\State Find the smallest integer $\ud{m}<m$ such that $t|d_{\ud{m}}^K|\le|d_m^K|$
\State If $|d_{\ud{m},m}^K|<\delta/\|u_i\|^2$, signal convergence and exit the current loop
\EndLoop
\State Obtain an approximation $\|u_i\|^2e_1^Tf(T_{\ud{m}})e_1$ of the unbiased sample $u_i^Tf(A)u_i$
\EndLoop
\State Obtain an estimate of $\tr(f(A))$ with $N$ samples
\State Obtain a confidence interval for the estimate by using Theorem~\ref{thm:bound}
\end{algorithmic}
\end{algorithm}

%%%%%%%%%%%%%%%%%%%%%%%%%%%%%%%%%%%%%%%%%%%%%%%%%%%%%%%%%%%%
\subsection{Parameters}\label{sec:param}
Two parameters related to the error estimation deserve some attention.

In principle, the Lanczos tolerance $\delta$ is free, because Theorem~\ref{thm:bound} is applicable to any positive $\delta$. In practice, it is not sensible to make $\delta$ too small or too large. In the former case the uncertainty in the statistical error dominates, whereas in the latter case the numerical bias dominates. A reasonable approach is to make these two sources of errors comparable; i.e., let $\delta=\beta\alpha s^{(m)}/\sqrt{N}$ for some $\beta\approx1$ (see Corollary~\ref{cor:bound}). The conundrum of this approach is that the sample standard error $s^{(m)}$ is unknown. To resolve the issue, one could run Algorithm~\ref{algo:hutchinson1} once as a precomputation, but omit all the unnecessary overheads. That is, no error estimation is performed, and the sample size $N'$ needs only be sufficient for the sample standard error to stabilize (e.g., $N'=30$), but it needs not be as large as $N$.

The number of poles, $K$, must be sufficiently large such that the error in the rational approximation of $f$ does not compromise the estimation of the bilinear form error. Based on Theorem~\ref{thm:rational.err}, we will require that the uniform error of $f-r_K$ be at most $\frac{1}{2}\delta/\|u_i\|^2$. When the random vectors $u_i$ are symmetric Bernoulli vectors, each $u_i$ has a constant 2-norm $\sqrt{n}$.

%%%%%%%%%%%%%%%%%%%%%%%%%%%%%%%%%%%%%%%%%%%%%%%%%%%%%%%%%%%%
\section{Experiments with 2D Laplacian}
To test the effectiveness of the  proposed method, we first verify the
several  algorithmic components  with the  2D Laplacian  matrix on  an
$n_1\times n_2$ grid:
\[
A=I_{n_2\times n_2}\otimes L_{n_1\times n_1} + L_{n_2\times n_2}\otimes I_{n_1\times n_1},
\]
where  $I$ is  the identity  matrix, $L$  is the  1D Laplacian  matrix
$\text{tridiag}(-1,2,-1)$,  and  the   subscripts  denote  the  matrix
size. This matrix is sparse and  is well suited for the Lanczos method
that  heavily relies on matrix-vector  multiplications.   Moreover,  its
eigenvalues  and  eigenvectors  are   known.  In  particular,  $A$  is
increasingly    ill    conditioned    (with   a    condition    number
$\approx4n/\pi^2$  for a  square  grid $n_1=n_2$)  and  the matrix  of
normalized  eigenvectors coincides  with the  matrix of  discrete sine
transform.  Hence,  the  ground  truth  $\tr(f(A))$  can  be  computed
economically, with an $O(n\log n)$ cost, through fast sine transform.

The experiments in this section consist of three parts: (a) the effectiveness of rational approximations for several commonly used functions $f$; (b) the effectiveness of error estimation for the bilinear form $v_1^Tf(A)v_1$; and (c) the effectiveness of the overall error estimation for $\tr(f(A))$ in the form of a confidence interval. Although rational approximations are not the contribution of this work, the purpose of part (a) is to obtain an empirical understanding of the needed number of poles, $K$.

%%%%%%%%%%%%%%%%%%%%%%%%%%%%%%%%%%%%%%%%%%%%%%%%%%%%%%%%%%%%
\subsection{Rational approximation}\label{sec:exp.lap2d.rat.approx}
We consider four functions with known fast-converging rational approximations: the negative exponential $\exp(-x)$, the square root $\sqrt{x}$, the logarithm $\log(x)$, and a composite of hyperbolic tangent and square root $\tanh(\sqrt{x})$, all used for $x>0$. These approximations are related to quadratures of contour integrals, as we briefly motivated in Section~\ref{sec:rat.approx}. The details for the exponential appear in Trefethen et al.~\cite{Trefethen2006}, who discussed approximations derived from both Talbot quadratures and best uniform approximations. The details for the latter three functions appear in Hale et al.~\cite{Hale2008}, who proposed using the trapezoid rule on conformal mappings of the circular contour. The resulting approximations in Hale et al.~\cite{Hale2008} are dependent on the spectrum interval of $A$.

Minor modifications are needed for our use. For the exponential, discussions in Trefethen et al.~\cite{Trefethen2006} are based on $\exp(x)$, $x\le0$; hence, we need to flip the sign of $x$ and accordingly negate the coefficients and poles, such that they agree with the canonical form~\eqref{eqn:rK}. Moreover, because both the coefficients and the poles come in conjugate pairs, we may keep only one from each pair, multiply the coefficients by $2$, and extract the real part of the sum. This results in the form~\eqref{eqn:rK2}, reducing the number of summation terms in~\eqref{eqn:rK} by half. We will use the best uniform approximation rather than Talbot quadratures because it converges twice as fast. For Matlab codes, see Figure 4.1 of Trefethen et al.~\cite{Trefethen2006}.

For the logarithm $\log(x)$ and the composite $\tanh(\sqrt{x})$, we
will use Method 2 and Method 1 of Hale et al.~\cite{Hale2008},
respectively. The formulas therein are in the form of
neither~\eqref{eqn:rK} nor~\eqref{eqn:rK2}: an additional
multiplicative term $A$ appears in the front and the imaginary part of
a summation is extracted instead of the real part. Hence, we turn to
the quadrature formula before the imaginary part is extracted, rewrite
the formula into the canonical form~\eqref{eqn:rK} plus a constant,
and extract the real part of the summation as done for the exponential
discussed above (which results in the same effect of reducing
summation terms by half).  The additional constant term attached to
the canonical form~\eqref{eqn:rK} cancels out when the quadrature is
used for approximating the bilinear form error (cf.~\eqref{eqn:dmK});
hence, it barely matters.

For the square root $\sqrt{x}$, we will use Method 3 
of Hale et al.~\cite{Hale2008}. No modifications are needed. Note that the poles are all on the negative real axis.

We plot  in Figure~\ref{fig:rational.appx} the error  $|f-r_K|$, where
recall that  $r_K$ is the  rational approximation with $K$  terms. For
the exponential, the interval is  $[0,8]$ and for the other functions,
the interval  is $[10^{-6},1]$. As can be seen, for the  exponential, a
very small number of points suffices  to decrease the uniform error to
approximately machine precision. For the other functions, $K$ needs to be
larger, but  often one  or a  few dozen points are sufficient.

\begin{figure}[ht]
\centering
\subfigure[$f(x)=\exp(-x)$]{
  \includegraphics[width=.48\linewidth]{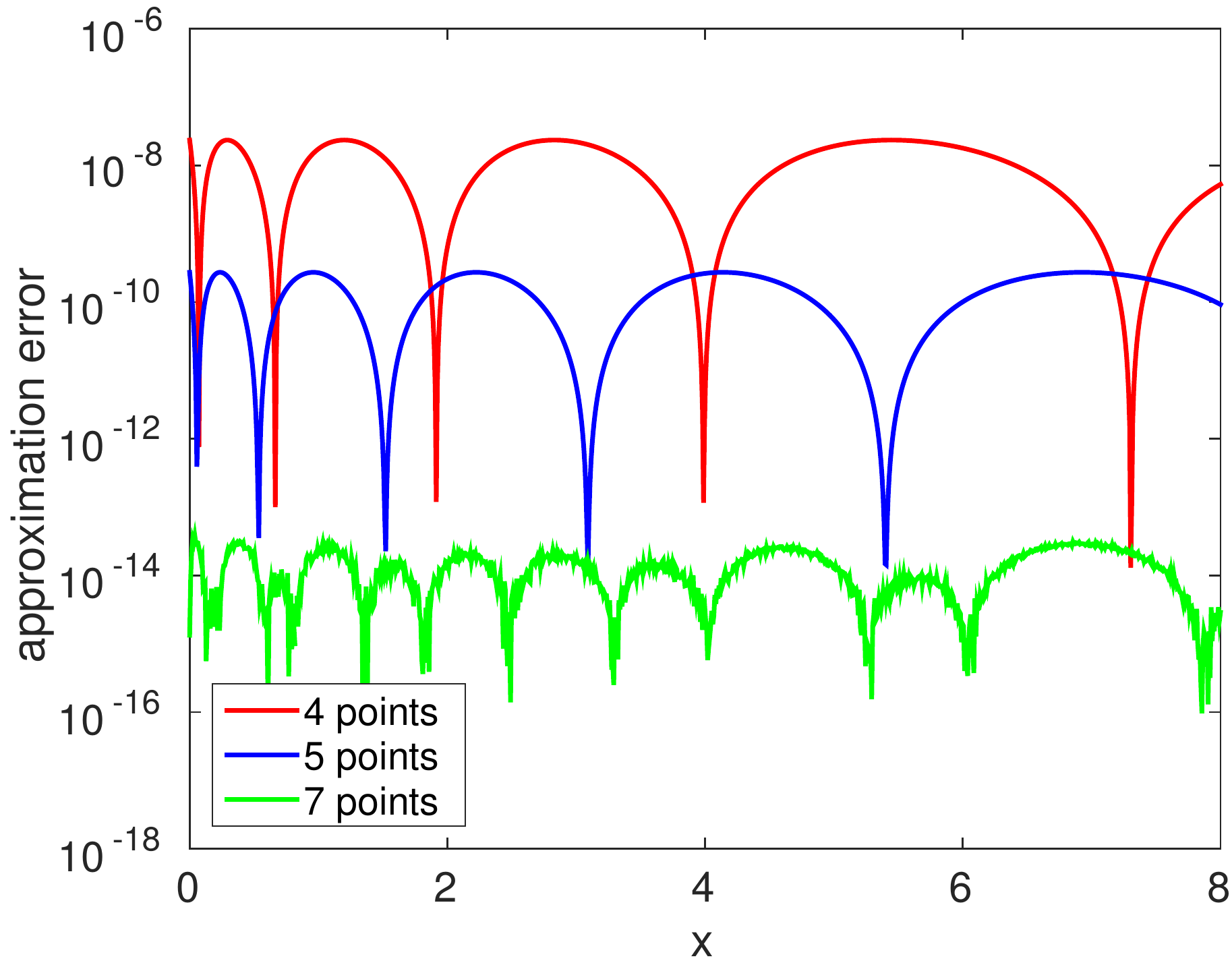}}
\subfigure[$f(x)=\sqrt{x}$]{
  \includegraphics[width=.48\linewidth]{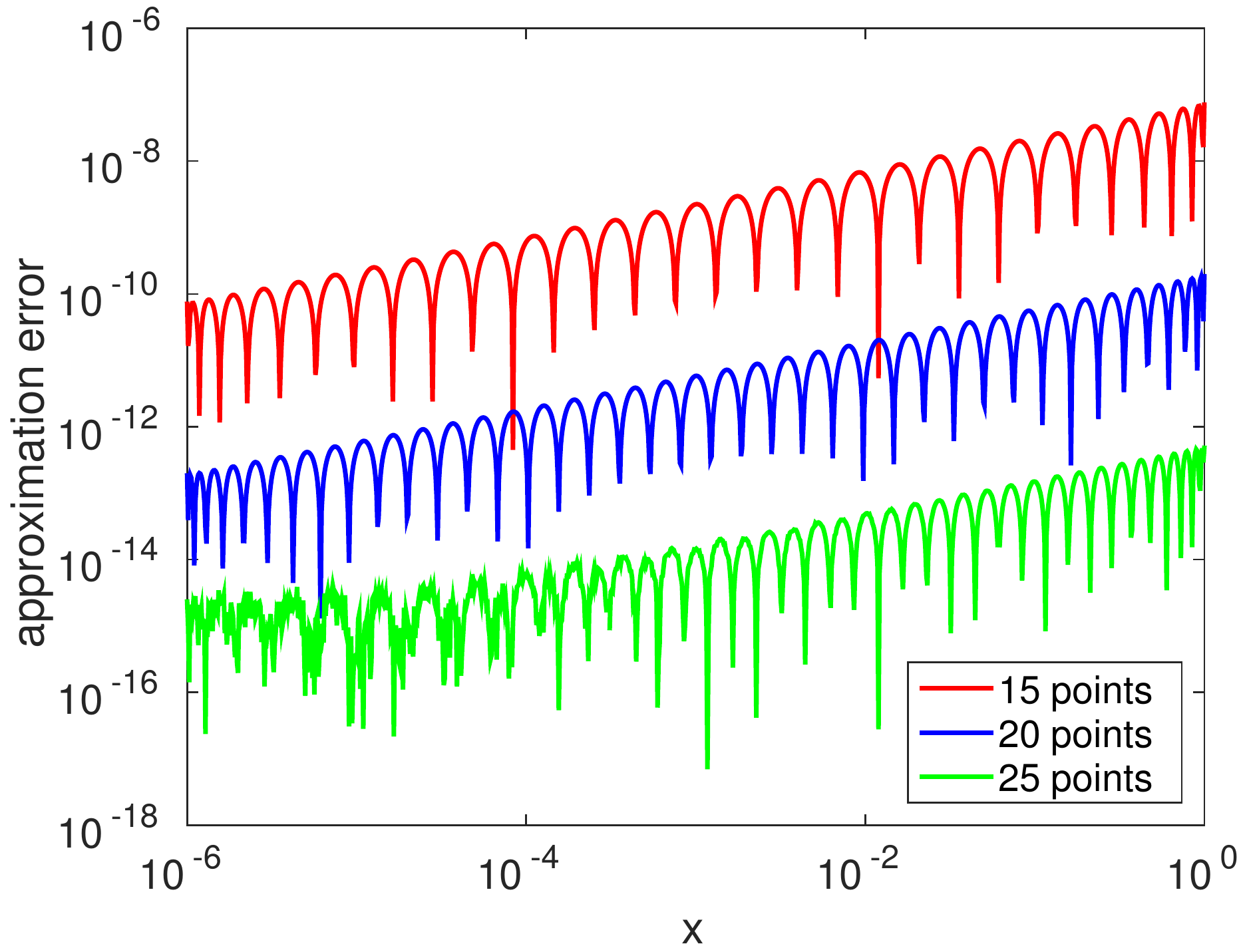}}\\
\subfigure[$f(x)=\log(x)$]{
  \includegraphics[width=.48\linewidth]{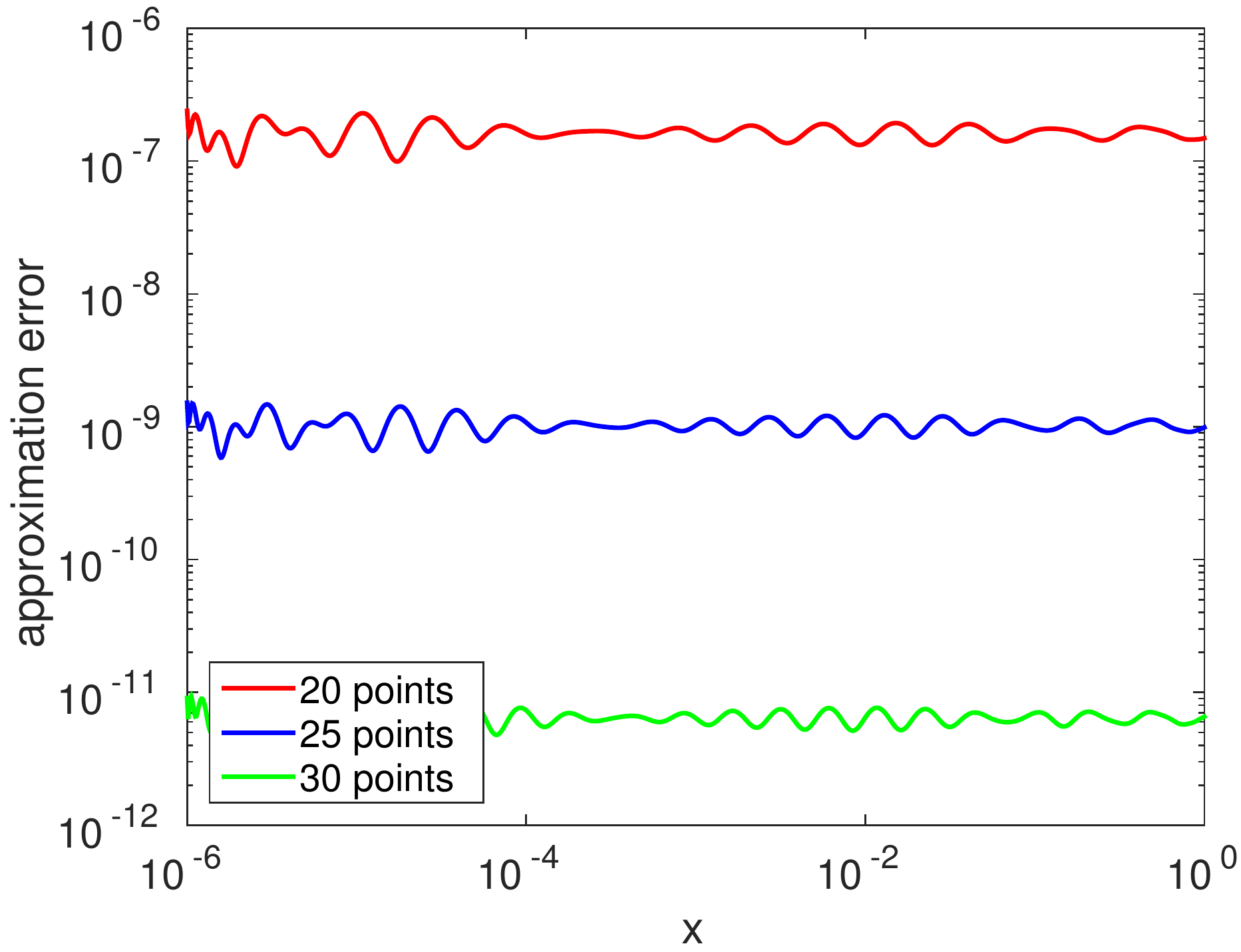}}
\subfigure[$f(x)=\tanh(\sqrt{x})$]{
  \includegraphics[width=.48\linewidth]{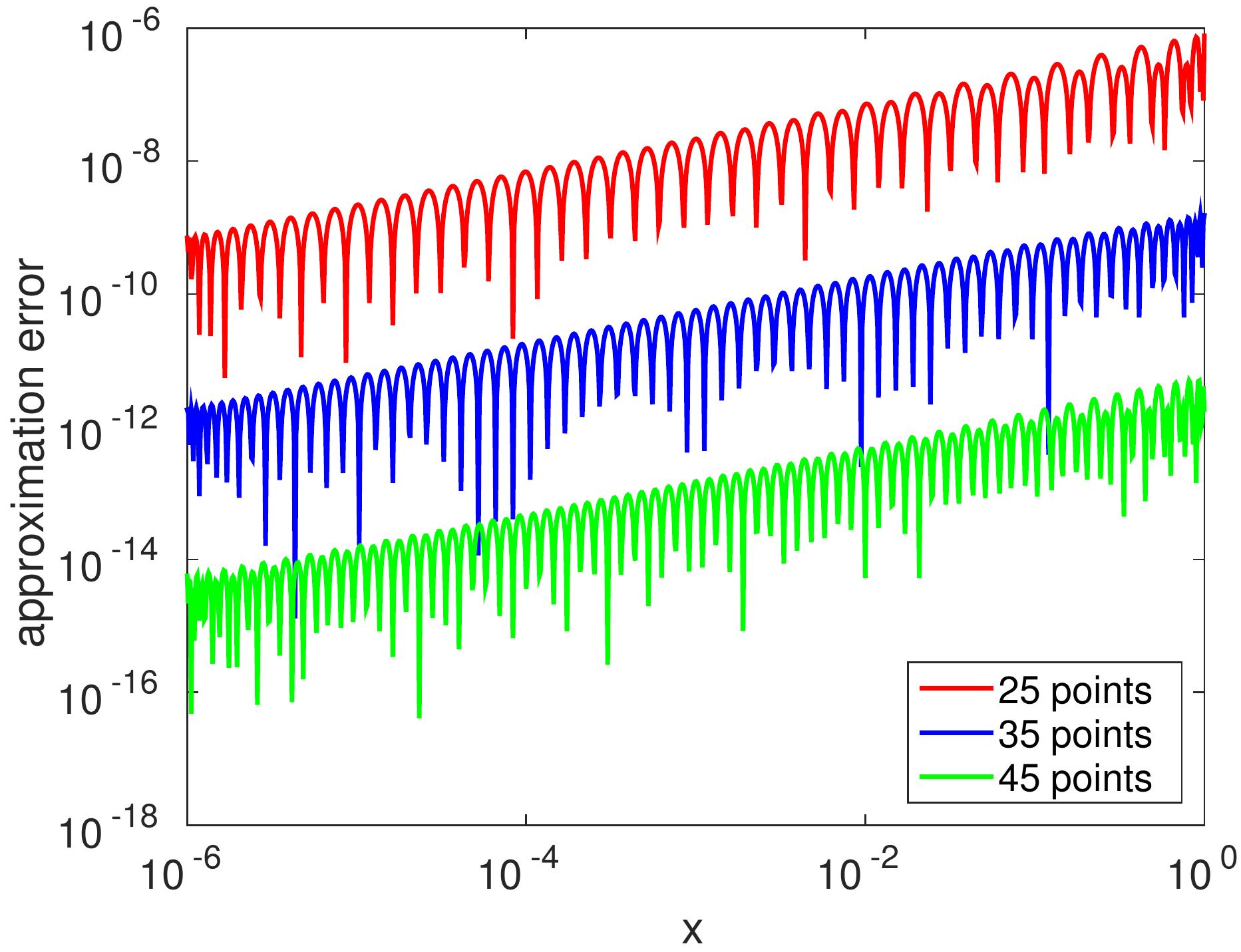}}
\caption{Rational approximation of various functions $f$ by using $K$ quadrature points.}
\label{fig:rational.appx}
\end{figure}

%%%%%%%%%%%%%%%%%%%%%%%%%%%%%%%%%%%%%%%%%%%%%%%%%%%%%%%%%%%%
\subsection{Error estimation for the bilinear form}\label{sec:exp.lap2d.bilinear.err}
We  use  a  $300\times400$  grid  as  an  example.  The  ground  truth
$v_1^Tf(A)v_1$ for any vector $v_1$ may be computed by using fast sine
transform,  as explained  earlier. Here,  we  choose $v_1$  to be  the
random vector of iid (independent and identically distributed) symmetric  Bernoulli variables, normalized to the
unit norm. The Lanczos approximation $e_1^Tf(T_m)e_1$ with $m$ Lanczos
steps is  then computed  and the  error is plotted  as the  blue solid
curve in Figure~\ref{fig:bilinear.form.err}.

\begin{figure}[ht]
\centering
\subfigure[$f(x)=\exp(-x)$]{
  \includegraphics[width=.48\linewidth]{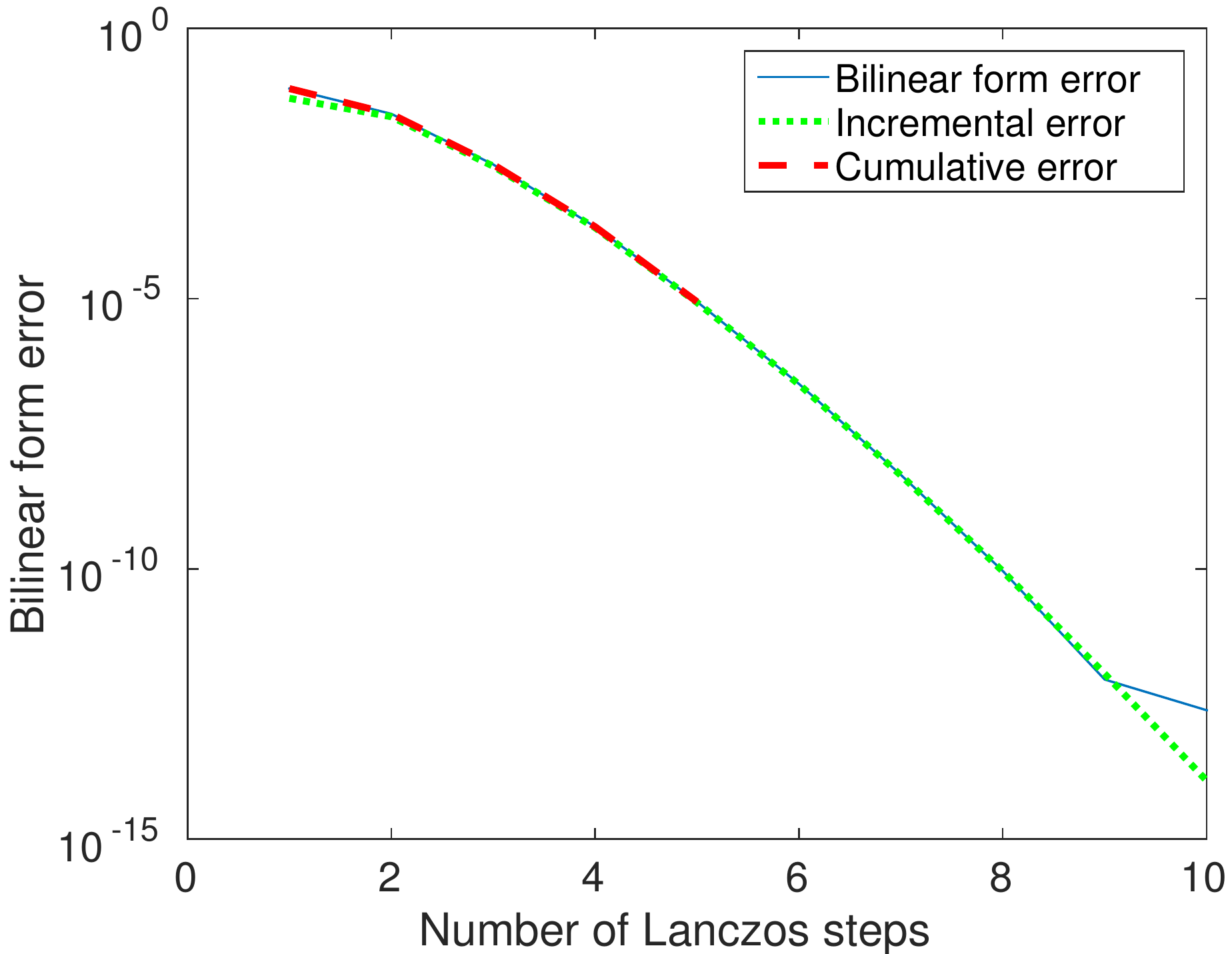}}
\subfigure[$f(x)=\sqrt{x}$]{
  \includegraphics[width=.48\linewidth]{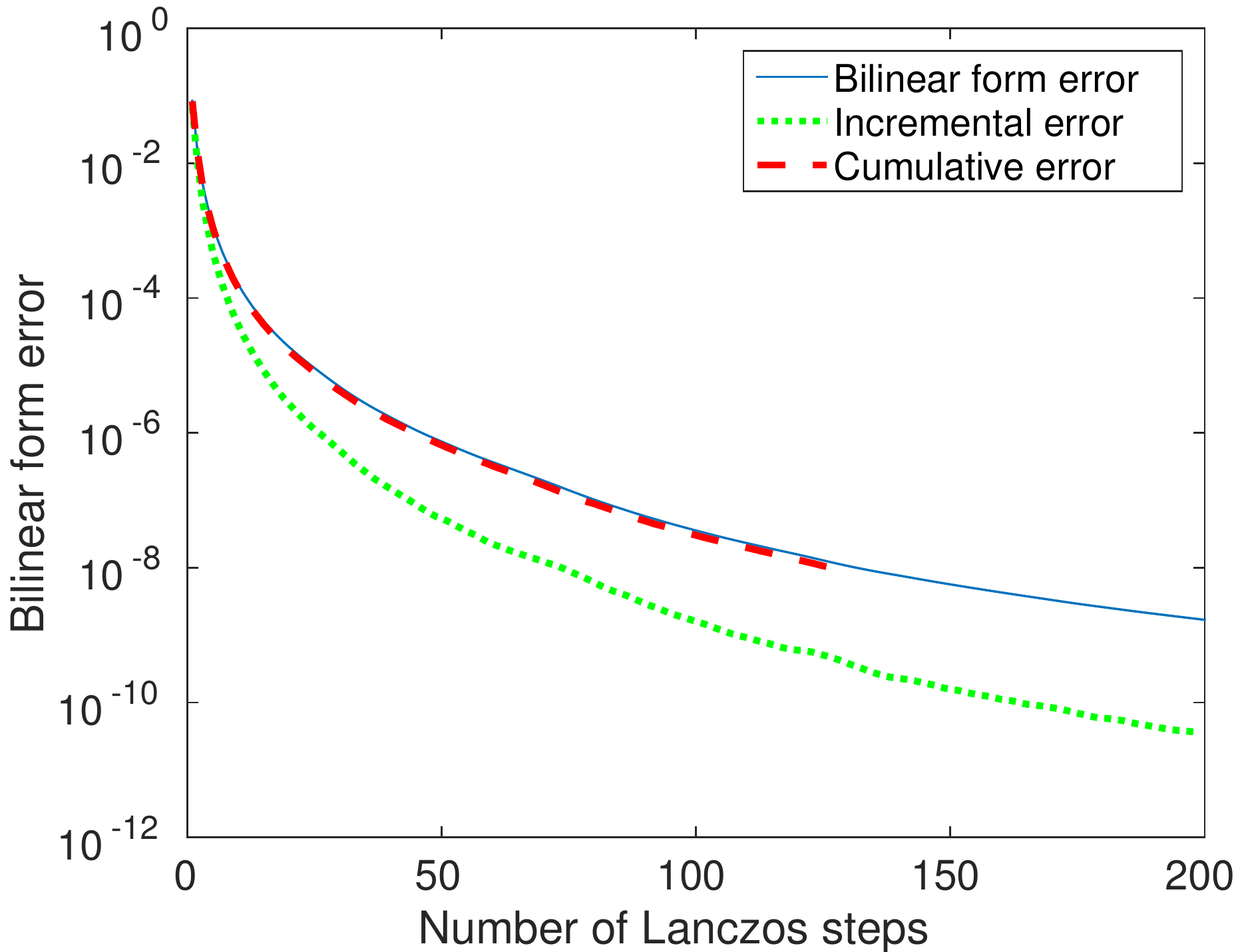}}\\
\subfigure[$f(x)=\log(x)$]{
  \includegraphics[width=.48\linewidth]{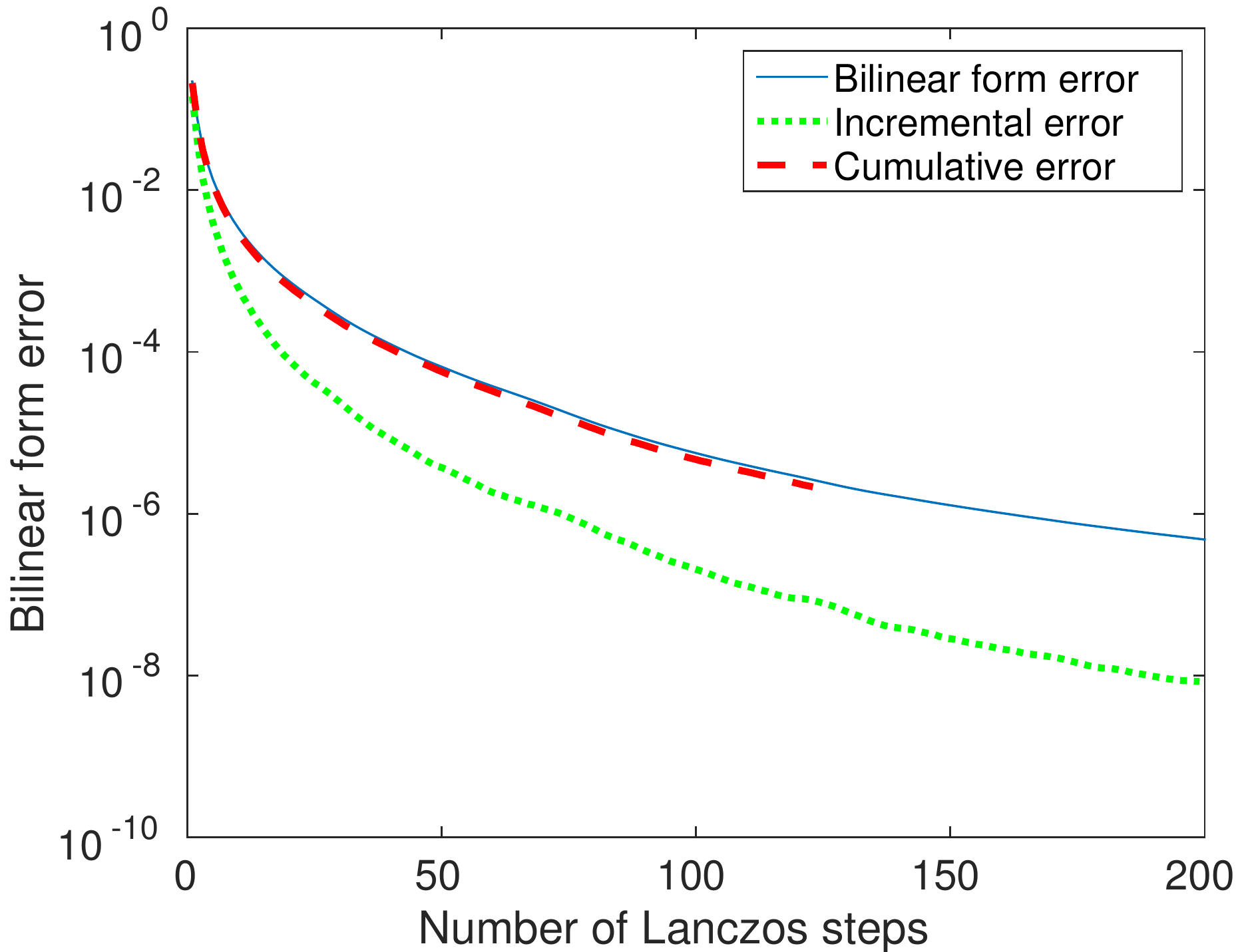}}
\subfigure[$f(x)=\tanh(\sqrt{x})$]{
  \includegraphics[width=.48\linewidth]{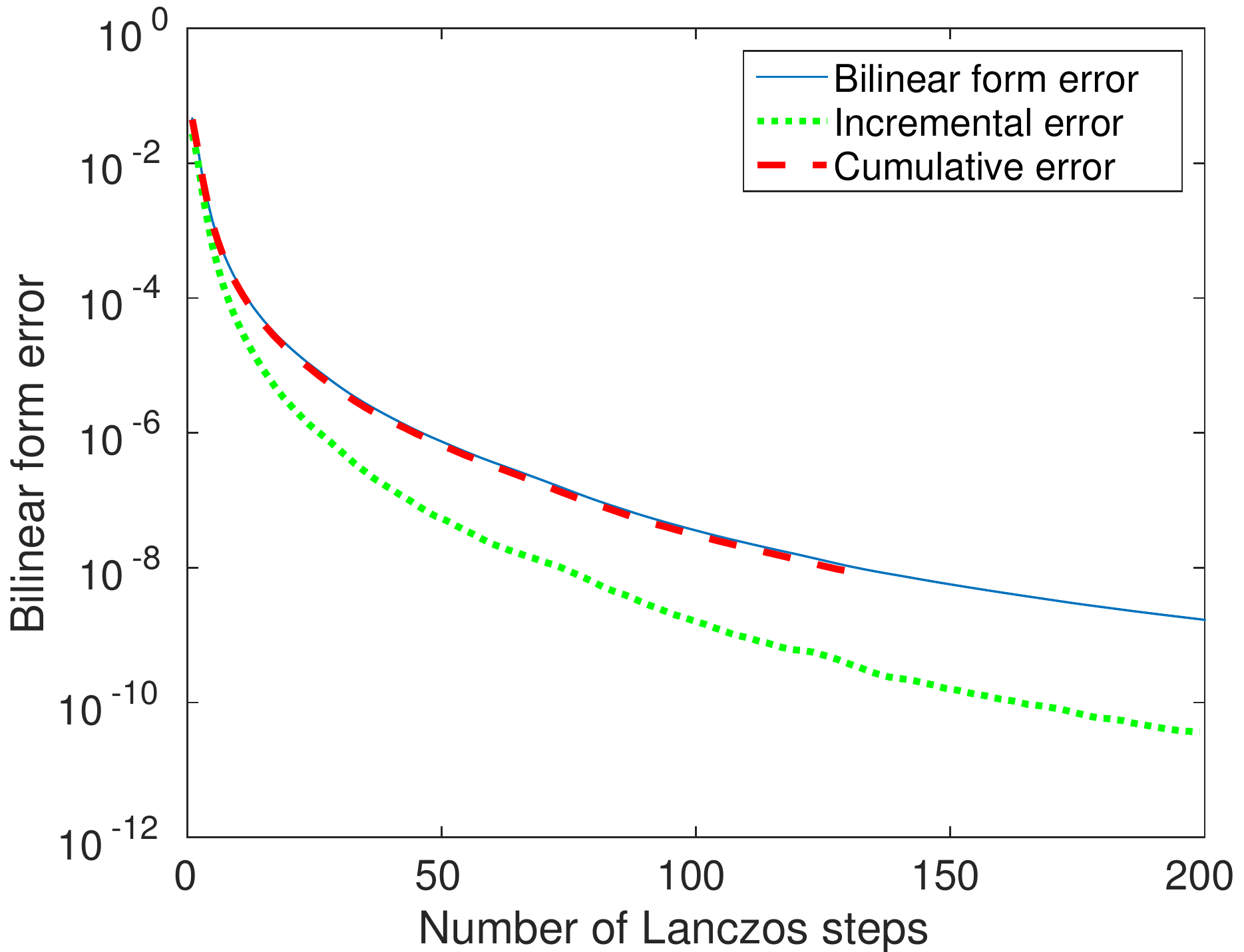}}
\caption{Bilinear form error and estimates, for 2D Laplacian matrix on a $300\times400$ grid.}
\label{fig:bilinear.form.err}
\end{figure}

To estimate this error, we compute the incremental error $d_m^K$ and the cumulative error $d_{m,m'}^K$, where $m'$ is the smallest integer greater than $m$ such that $|d_{m'}^K|/|d_m^K|\le t$.
The absolute value of these errors is plotted as the green and red dashed curves in the figure, respectively. Clearly, we may plot only the $d_{m,m'}^K$'s that satisfy $m'\le200$, which is the maximum number of Lanczos steps seen in Figure~\ref{fig:bilinear.form.err}.

As  can  be  seen,  for  the  exponential,  the  three  errors  nearly
overlap. It is for  this reason that we do not  plot the whole red
curve; otherwise,  it fully covers  the green curve. The  extremely fast
convergence implies that the incremental error suffices as an estimate
of the bilinear  form error. For the other  functions, the incremental
error is far  from the bilinear form error, and  hence it is necessary
to do an accumulation to get a better estimate. The fact  that the curve of
cumulative errors nearly overlaps with that of the bilinear form errors
indicates              that              the              accumulation
criterion is effective.

\subsection{Overall error estimation with confidence interval}
\label{sec:exp.lap2d.overall}

With the  preparation of the  preceding two subsections, we  now apply
Theorem~\ref{thm:bound}  to  establish  confidence intervals  for  the
approximation of  $\tr(f(A))$. To this end,  we fix the number  $N$ of
random vectors  to be $100$ and  set $\alpha=3$, which ensures  a high
probability $p_{\alpha}\approx99.73\%$. We vary the size of the matrix
$A$  (and  hence the  condition  number)  by using  progressively
larger grids. The setting of the number of quadrature points, $K$, and the Lanczos tolerance $\delta$ follows Section~\ref{sec:param}.

We   perform   the  computations   and   summarize   the  results   in
Tables~\ref{tab:estim.with.conf.intv}
and~\ref{tab:estim.with.conf.intv.2}. As the  grid becomes larger, the
condition number  of $A$  increases, which
  leads to a larger $\delta$  and $K$. Interestingly, for the largest
grid (which  corresponds to  $n\approx10^6$), $K=3$  quadrature points
are sufficient for the exponential,  and for other functions, $K$ does
not  exceed  two dozens.  Then,  the  resulting accuracy  of the rational
approximation is five to six digits. The number of Lanczos steps, $m$,
is as small  as $6$ for the  exponential and no greater  than $34$ for
the  logarithm,  on  average.  Moreover,  the  approximated  trace  is
generally three-  to four-digit  accurate, and  the half-width  of the
confidence interval is  generally a few times  the  actual error (in
several cases,  mostly for large problems,  it is less than  twice the
actual  error). As  expected, the  time  for estimating  the error  is
negligible compared with that for approximating the trace.

Note that for this problem, the implementation of the Lanczos
algorithm does not affect timing much, even using full
  reorthogonalization. It turns out that the accuracies are barely
  affected by the loss of orthogonality, possibly because 2D
  Laplacians are easy to handle. For later experiments, however,
  reorthogonalization is crucial because the approximation error
  substantially degrades without it (see an illustration in the
  appendix). In these experiments, the matrix may be much larger and
  Lanczos converges more slowly. Hence, to gain time efficiency, it
%%YS:  ..... is imperative to speed up 
will be  beneficial to replace  the simple full reorthogonalization therein
by  a more sophisticated scheme such as partial
reorthogonalization~\cite{Simon1984,Fang2012}.  Therefore, we
implemented and used partial reorthogonalization for all experiments
in this paper. For more details on the implementation and the machine
setting, see the next section.

\begin{table}[ht]
\caption{Approximately computing $\tr(f(A))$ with $99.73\%$ confidence interval for 2D Laplacian matrix $A$. In all cases, the number $N$ of random vectors is $100$.}
\label{tab:estim.with.conf.intv}
\centering\tt

%\vskip10pt
\begin{tabular}{>{\rm}lccc}
$f(x)=\exp(-x)$ \\
\hline
Grid size & 90$\times$120 & 300$\times$400 & 900$\times$1200\\
\hline
\# Quadrature points, $K$           & 2          & 3         & 3\\
Rational approx. error              & 1.72e-04   & 2.01e-06  & 2.01e-06\\
\hline
Lanczos tolerance $\delta$          & 8.31       & 26.1      & 71\\
Average \# of Lan. steps, $m$       & 5          & 5         & 6\\
\hline
Truth $\mu=\tr(f(A))$               &    1014.96 &   11378.0 &   102662\\
Approximation result $\bar{x}^{(m)}$ &    1016.38 &   11367.3 &   102630\\
99.73\% Confidence interval         & $\pm$19.14 & $\pm$60.1 & $\pm$164\\
\hline
Time approximation (seconds)        & 0.28       & 2.12      & 21.75\\
Time error estimate (seconds)       & 0.01       & 0.02      &  0.04\\
\hline
\end{tabular}

\vskip10pt
\begin{tabular}{>{\rm}lccc}
$f(x)=\sqrt{x}$ \\
\hline
Grid size & 90$\times$120 & 300$\times$400 & 900$\times$1200\\
\hline
\# Quadrature points, $K$           & 6         & 8        & 10\\
Rational approx. error              & 2.71e-04  & 9.65e-05 & 3.99e-05\\
\hline
Lanczos tolerance $\delta$          & 25.1      & 80       & 220\\
Average \# of Lan. steps, $m$       & 5.04      & 7.07     & 10.01\\
\hline
Truth $\mu=\tr(f(A))$               &   20708.0 &   229986 & 2.06961e+06\\
Approximation result $\bar{x}^{(m)}$ &   20715.9 &   230071 & 2.06984e+06\\
99.73\% Confidence interval         & $\pm$57.7 & $\pm$185 & $\pm$507\\
\hline
Time approximation (seconds)        & 0.43      & 4.56     & 53.39\\
Time error estimate (seconds)       & 0.01      & 0.04     &  0.10\\
\hline
\end{tabular}
\end{table}

\begin{table}[ht]
\caption{(Continued from Table~\ref{tab:estim.with.conf.intv}) Approximately computing $\tr(f(A))$ with $99.73\%$ confidence interval for 2D Laplacian matrix $A$. In all cases, the number $N$ of random vectors is $100$.}
\label{tab:estim.with.conf.intv.2}
\centering\tt

%\vskip10pt
\begin{tabular}{>{\rm}lccc}
$f(x)=\log(x)$ \\
\hline
Grid size & 90$\times$120 & 300$\times$400 & 900$\times$1200\\
\hline
\# Quadrature points, $K$           & 9         & 10       & 14\\
Rational approx. error              & 2.82e-04  & 6.56e-04 & 4.64e-05\\
\hline
Lanczos tolerance $\delta$          & 38.0      & 120      & 314\\
Average \# of Lan. steps, $m$       & 10.16     & 18.19    & 33.29\\
\hline
Truth $\mu=\tr(f(A))$               &   12652.9 &   140146 & 1.26014e+06\\
Approximation result $\bar{x}^{(m)}$ &   12672.4 &   140319 & 1.26060e+06\\
99.73\% Confidence interval         & $\pm$87.5 & $\pm$277 & $\pm$723\\
\hline
Time approximation (seconds)        & 0.90      & 13.11    & 194.38\\
Time error estimate (seconds)       & 0.04      &  0.16    &   0.49\\
\hline
\end{tabular}

\vskip10pt
\begin{tabular}{>{\rm}lccc}
$f(x)=\tanh(\sqrt{x})$ \\
\hline
Grid size & 90$\times$120 & 300$\times$400 & 900$\times$1200\\
\hline
\# Quadrature points, $K$           & 12         & 15       & 20\\
Rational approx. error              & 6.84e-05   & 3.68e-05 & 9.77e-06\\
\hline
Lanczos tolerance $\delta$          & 5.73       & 18       & 48\\
Average \# of Lan. steps, $m$       & 8.00       & 11.25    & 16.17\\
\hline
Truth $\mu=\tr(f(A))$               &    9928.62 &  110240  &   991960\\
Approximation result $\bar{x}^{(m)}$ &    9930.14 &  110261  &   992025\\
99.73\% Confidence interval         & $\pm$13.13 & $\pm$41  & $\pm$110\\
\hline
Time approximation (seconds)        & 0.66       & 6.49     & 83.08\\
Time error estimate (seconds)       & 0.03       & 0.08     &  0.19\\
\hline
\end{tabular}
\end{table}

%%%%%%%%%%%%%%%%%%%%%%%%%%%%%%%%%%%%%%%%%%%%%%%%%%%%%%%%%%%%
\section{Experiments with covariance matrices}\label{sec:cov}
In  this  section, we  present  experiments  with covariance  matrices
encountered              in              Gaussian              process
analysis~\cite{Stein1999,Rasmussen2006,GP-SAA,score.func}.  A Gaussian
process is a  stochastic process with Gaussian  properties. Central to
the mathematical tool is a covariance kernel function that generates a
covariance  matrix  $A$ for  sampling  sites,  where the  observations
collectively follow a multivariate normal distribution with covariance
$A$. Many  tasks, including hyperparameter estimation  and prediction,
require  a computation  with the  matrix $A$.  Here, we  focus on  the
log-determinant  term that  appears  in  the Gaussian  log-likelihood,
which needs to be optimized  for estimating the hyperparameters of the
process. Clearly, for a symmetric positive-definite matrix $A$,
\[
\log\det(A)=\tr(\log(A)).
\]

For demonstration, we will use the Mat\'{e}rn kernel function plus a nugget
\[
\phi(r)=\frac{(\sqrt{2\nu}r)^{\nu}\besselk_{\nu}(\sqrt{2\nu}r)}{2^{\nu-1}\Gamma(\nu)}
+\tau\cdot\delta(r=0)
\quad\text{with}\quad
r=\sqrt{\sum_{i=1}^d\frac{(x_i-x'_i)^2}{\ell_i^2}}
\]
as  an  example.  Here  $x=[x_1,x_2,\ldots,x_d]$  denotes  a  site  in
$\real^d$, $r$ denotes  the elliptical distance between  two sites $x$
and  $x'$  with  elliptical  scaling  $[\ell_1,\ell_2,\ldots,\ell_d]$,
$\besselk_{\nu}$ is the modified Bessel function of the second kind of
order  $\nu$, $\Gamma(\nu)$ is the Gamma function, $\delta(r=0)$  is the  Kronecker delta  taking $1$  when
$r=0$ and  $0$ otherwise, and  $\tau$ is the  size of the  nugget. The
Mat\'{e}rn  kernel (even  without  the  nugget) is  strictly positive-definite,
meaning that  the generated matrix  $A=[\phi(x-x')]$ for all  pairs of
sites  $x$ and  $x'$ is  positive-definite. The  Mat\'{e}rn kernel  is
even, achieves its maximum $1$ at  the origin, and monotonically decreases
when $r>0$.

We assume that the sites are located on a regular grid of size $n_1\times n_2$ and set the scaling parameters to be $\ell_1=0.4\times n_2$ and $\ell_2=0.4\times n_1$. We also set the smoothness parameter $\nu=1.5$ and the nugget $\tau=10^{-5}$. With a regular grid structure, matrix-vector multiplications with $A$ has an $O(n)$ memory and $O(n\log n)$ time cost, although $A$ is fully dense, because the multiplications may be done through circulant embedding followed by fast Fourier transform (FFT)~\cite{Chan2007,cg.fft.journal}.

To make  the experiment more interesting,  we let the sites  be $10\%$
uniformly  random samples  of  the  grid (i.e.,  the  number of  sites
$n=0.1\times n_1n_2$).  Hence, strictly speaking, the  sites no longer
form  a  regular  grid;  they  are  \emph{scattered}  sites.  However,
matrix-vector multiplications may still be performed through circulant
embedding and  FFT, because  of the  underlying grid  structure. Note,
nevertheless, that the cost  
is not reduced by a factor of $10$ as is the case for the number
of sites.

Unlike  the  2D  Laplacian  in the  preceding  section,  the  spectral
information  of the  covariance  matrix is  only  partially known.  In
particular, we  know that the smallest  eigenvalue of $A$ has  a lower
bound $\tau$ (the  nugget) but do not know the  largest eigenvalue. In
theory, the largest eigenvalue grows approximately proportionally with
$n$   and  the   smallest  eigenvalue   decreases  to   $\tau$  fairly
quickly~\cite{lap.precond}. Hence, we  estimate the largest eigenvalue
by using  the Lanczos  method and  set the lower  end of  the spectrum
interval to  be $\tau$. By  using this  spectrum interval we  obtain a
rational approximation of $f(x)=\log(x)$, needed for error estimation.

As mentioned at  the end of the preceding subsection,  we implemented the
Lanczos        iteration~\eqref{eqn:lanczos}       with        partial
reorthogonalization~\cite{Simon1984,Fang2012}. The reason  is that, as
will be seen soon,  the number of Lanczos steps is  no longer as small
as those  in the case of  2D Laplacian and so  the reorthogonalization
cost is quite  high if full reorthogonalization is  used.

The program is written in Matlab and run on a laptop with eight Intel cores (CPU frequency 2.8GHz) and 32GB memory. By default, Matlab uses four threads in many built-in functions, but we observe that at most two cores are active during the computation.

In Table~\ref{tab:estim.with.conf.intv.covar}, we summarize the
computation results for varying grid sizes from $160\times90$ to
$1600\times900$.  For the first two grids, performing spectral
decomposition is affordable and hence we also compute the ground-truth
condition numbers and log-determinants. As can be seen, using the
nugget $\tau$ (lower bound) as an estimate of the smallest eigenvalue
suffices for indicating the magnitude of the condition number.  As
expected, the condition number grows approximately by a factor of $10$
every time we increase the grid size by this factor.  With an
increasing condition number, the log-determinant is harder to compute,
requiring more Lanczos steps. Note that the scale of this number---in
the hundreds---is much larger than that for the 2D Laplacian. Taking
another factor into account, namely the matrix size, it takes quite
some time to finish the computation (for the largest grid, several
hours), although the costs of the spectrum estimation and error
estimation are negligible.  The benefit, on the other hand, is that we
have a useful error bound for the approximated trace, which gives a
confidence in the computation which would have been impossible without
a reliable error estimate.

\begin{table}[ht]
\caption{Approximately computing $\log\det(A)$ with $99.73\%$ confidence interval for covariance matrix $A$. In all cases, the number $N$ of random vectors is $100$.}
\label{tab:estim.with.conf.intv.covar}
\centering\tt
%\vskip10pt
\begin{tabular}{>{\rm}lccc}
\hline
Grid size & 160$\times$90 & 500$\times$300 & 1600$\times$900\\
\hline
Condition number (truth)            & 4.08e+07  & 5.54e+08 & ---\\
Condition number (estimated)        & 5.17e+07  & 5.60e+08 & 5.22787e+09\\
\hline
\# Quadrature points, $K$           & 12        & 15       & 18\\
Rational approx. error              & 5.19e-03  & 1.39e-03 & 4.18e-04\\
\hline
Lanczos tolerance $\delta$          & 40.5      & 99       & 288\\
Average \# of Lan. steps, $m$       & 103       & 240      & 425\\
\hline
Truth $\mu=\log\det(A)$             &  -10844.7 & -151826  & ---\\
Approximation result $\bar{x}^{(m)}$ &  -10794.3 & -151715  & -1.60122e+06\\
99.73\% Confidence interval         & $\pm$92.6 & $\pm$228 & $\pm$480\\
\hline
Time spectrum estim. (seconds)      &  0.04     &   0.3    &     3\\
Time trace approx. (seconds)        & 26.87     & 781.6    & 14253\\
Time error estimate (seconds)       &  0.43     &   1.9    &     3\\
\hline
\end{tabular}
\end{table}

%%%%%%%%%%%%%%%%%%%%%%%%%%%%%%%%%%%%%%%%%%%%%%%%%%%%%%%%%%%%
\section{Concluding remarks}
In  this  work,  we  proposed  two  error  estimates  related  to  the
computation  of $f(A)$:  one for  its bilinear  form and  one for  its
trace. The bilinear form  is a building block of the  trace in a Monte
Carlo-type   approximation.   We   focused   on   the   symmetric
positive-definite  case  for  $A$,  where the Lanczos algorithm  has  long  been  a
preferred iterative  method and where a  representative application is
the  covariance matrix,  whose log-determinant  (i.e., $f(x)=\log(x)$)
constitutes a significant computational  component of Gaussian process
analysis.

The bilinear form $v_1^Tf(A)v_1$  is approximated by $e_1^Tf(T_m)e_1$,
where $v_1$,  if normalized,  is the starting  vector of  Lanczos, and
$T_m$  is  the   tridiagonal  matrix  resulting  from   $m$  steps  of
the Lanczos process.  The  approximation  error  is  gauged  through  economically
accumulating                     incremental                    errors
$d_k=e_1^Tf(T_{k+1})e_1-e_1^Tf(T_k)e_1$    for    $k=m,m+1,m+2\ldots$,
because eventually  $v_1^Tf(A)v_1$ is nothing but  $e_1^Tf(T_n)e_1$ if
$A$ has a size $n\times n$. The challenging question is how many terms
one should accumulate  in order to obtain a  reasonable, yet economic,
estimate  of  the approximation  error,  because  these terms  require
running extra Lanczos  iterations beyond the $m$-th  one. Our proposal
is that one should accumulate  till $k=m'-1$, where $|d_{m'}/d_m|$ falls
under  some threshold  $t$, a  reasonable choice  being $0.1$.  Such a
proposal is motivated by two facts: (i) if the even derivatives of $f$
have the same sign, then so are the incremental errors $d_k$; and (ii)
if $f$  is analytic in the  spectrum interval of $A$  and analytically
continuable inside  an open Bernstein  ellipse whose foci are  the two
ends  of  the interval,  then  Lanczos  converges exponentially.  Many
functions  in practical  applications  for positive-definite  matrices
incidentally meet these two  criteria, including those used in our
experiments:    $\exp(-x)$,    $\log(x)$,   and    $x^{\alpha}$    for
$\alpha<1$.  This  leaves room  for  future  investigations as to whether  the
proposal is more generally applicable  to functions without  the constant-sign
property for its even derivatives.

In   retrospect,   our  proposal   resembles   that   of  Frommer   et
al.~\cite{Frommer2013} in several ways. Even though their work focused
on  rational functions,  our work  in effect  also relies  on rational
approximations  of a  general  function.  Their \emph{Lanczos  restart
  recovery} strategy corresponds to our  idea of running extra Lanczos
iterations. A crucial distinction, 
however, is that  the number of extra Lanczos iterations
in   our    case   is determined implicitly   by    the   requirement
$|d_{m'}/d_m|\le0.1$, whereas this number  needs be
prescribed in advance in Frommer et al.~\cite{Frommer2013}
The benefit  of an  implicit determination  over a  prescribed
one  is  a   much  more  accurate  estimate,   independent  of  the
convergence speed. On the other hand,
prescribing  a good number of extra Lanczos steps for  sharp estimates
likely requires an a priori knowledge of the convergence behavior.

The second  contribution of our  work is  the error estimation  of the
trace. The trace  $\tr(f(A))$ is approximated by a  Monte Carlo sample
average of  $u_i^Tf(A)u_i$, where  the $u_i$'s are  independent random
vectors  and  each  sample   $u_i^Tf(A)u_i$  is  unbiased.  Therefore,  the
approximation error  generally follows basic estimation  theory, where
confidence intervals are established as a means for bounding the error
in a (high) probability. The distinction, however, is that the samples
(the bilinear  forms) are only  approximately computed and  hence they
bear a numerical bias. Our strategy  is to impose a bound $\delta$ (as
a  stopping criterion  for  the bilinear  form  approximation) on  the
numerical bias and inject $\delta$  into the confidence interval. Such
a treatment  is quite  general, overcoming the  limitation of  a prior
work~\cite{Hutchinson} applicable to  only multivariate normal vectors
$u_i$. A  restriction, on  the other  hand, is that  the setting  of a
reasonable $\delta$ is relatively  blind before computation, whereas
the  prior  work proposes  directly  the  tolerance that  matches  the
numerical error with the statistical error.
If one wants a $\delta$  that makes these
two sources  of errors  comparable in  the framework  of this  work, a
practical  approach  is  to  
run a precomputation (the same as Algorithm~\ref{algo:hutchinson1} but without error estimate),  get  an
approximate  standard  error of the  samples,  and  follow
Corollary~\ref{cor:bound}.

%%%%%%%%%%%%%%%%%%%%%%%%%%%%%%%%%%%%%%%%%%%%%%%%%%%%%%%%%%%%
\section*{Acknowledgments}
We are thankful to the anonymous referees whose comments help improve
the paper. In particular, one referee pointed out an alternative
derivation of the algorithm for computing the incremental error,
mentioned at the end of Section~\ref{sec:incre.err}.  Jie Chen was
supported by the XDATA program of the Defense Advanced Research
Projects Agency (DARPA), administered through Air Force Research
Laboratory contract FA8750-12-C-0323.  Yousef Saad was supported by
NSF grant CCF-1318597.

%%%%%%%%%%%%%%%%%%%%%%%%%%%%%%%%%%%%%%%%%%%%%%%%%%%%%%%%%%%%
\appendix
\section{Effects of loss of orthogonality}
We illustrate in Figure~\ref{fig:orth} that the Lanczos convergence for $v_1^T\log(A)v_1$ substantially degrades without reorthogonalization.

\begin{figure}[ht]
\centering
\includegraphics[width=.48\linewidth]{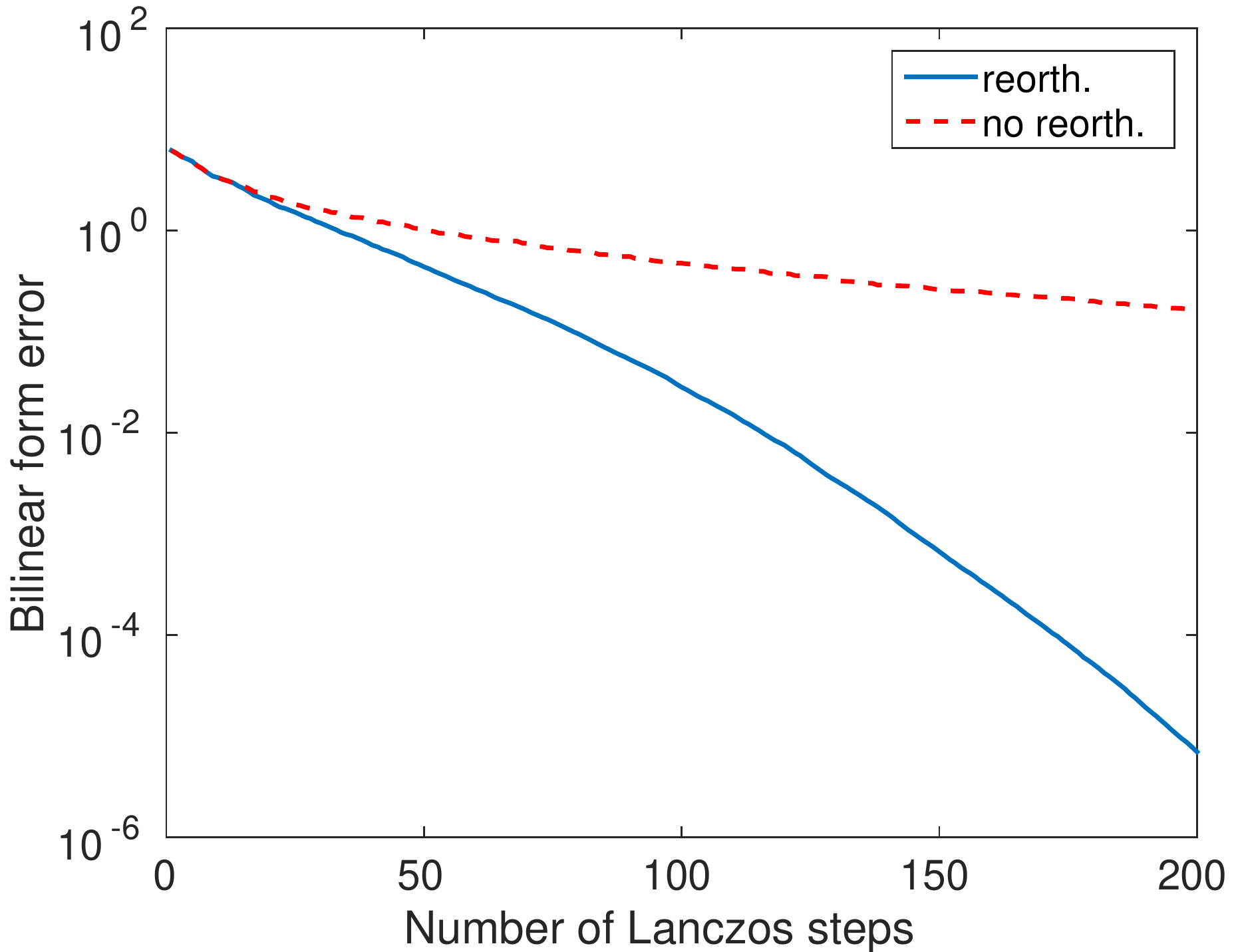}
\caption{Convergence history of $v_1^T\log(A)v_1$ for $A$ defined in Section~\ref{sec:cov} on a $90\times120$ grid.}
\label{fig:orth}
\end{figure}

%%%%%%%%%%%%%%%%%%%%%%%%%%%%%%%%%%%%%%%%%%%%%%%%%%%%%%%%%%%%
\bibliographystyle{plain}
\bibliography{reference}

\end{document}